\documentclass[11pt,letterpaper]{amsart}

\usepackage{amsmath}
\usepackage{amssymb}
\usepackage{amsfonts}
\usepackage{mathrsfs}
\usepackage{array}
\usepackage{enumerate}
\usepackage[french, english]{babel}
\usepackage[all]{xy}
\usepackage[T1]{fontenc}
\usepackage{float}
\usepackage{dsfont}

\numberwithin{equation}{section}

\newcommand{\msg}{\mathscr{G}}
\newcommand{\msi}{\mathscr{I}}
\newcommand{\msp}{\mathscr{P}}
\newcommand{\msq}{\mathscr{Q}}

    \newcommand{\BC}{{\mathbb {C}}} 
     \newcommand{\BF}{{\mathbb {F}}}
     \newcommand{\BH}{{\mathbb {H}}}

     \newcommand{\BP}{{\mathbb {P}}}
    \newcommand{\BQ}{{\mathbb {Q}}} \newcommand{\BR}{{\mathbb {R}}}
     \newcommand{\BT}{{\mathbb {T}}}

     \newcommand{\BZ}{{\mathbb {Z}}}

    \newcommand{\CE}{{\mathcal {E}}} \newcommand{\CF}{{\mathcal {F}}}
    \newcommand{\CG}{{\mathcal {G}}} 
     
    \newcommand{\CK}{{\mathcal {K}}} 
     
    \newcommand{\CO}{{\mathcal {O}}} 
     
    \newcommand{\CS}{{\mathcal {S}}}

     \newcommand{\fb}{{\mathfrak{b}}}
     \newcommand{\fd}{{\mathfrak{d}}}

     \newcommand{\fl}{{\mathfrak{l}}}

     \newcommand{\ft}{{\mathfrak{t}}}

     \newcommand{\fA}{{\mathfrak{A}}} \newcommand{\fB}{{\mathfrak{B}}}
     
     \newcommand{\fF}{{\mathfrak{F}}}
    \newcommand{\fG}{{\mathfrak{G}}} 
     
    \newcommand{\fK}{{\mathfrak{K}}} \newcommand{\fL}{{\mathfrak{L}}}
    \newcommand{\fM}{{\mathfrak{M}}}

    \newcommand{\fS}{{\mathfrak{S}}}

    \renewcommand{\Im}{{\mathrm{Im}}}

    \newcommand{\ord}{{\mathrm{ord}}}

    \renewcommand{\mod}{\ \mathrm{mod}\ }\renewcommand{\Re}{{\mathrm{Re}}}

    \newcommand{\Sel}{{\mathrm{Sel}}}

    \font\cyr=wncyr10

    \newcommand{\Sha}{\hbox{\cyr X}}

    \font\cyr=wncyr10

    \theoremstyle{plain}
    \newtheorem{thm}{Theorem}[section] 
    \newtheorem{lem}[thm]{Lemma}  \newtheorem{prop}[thm]{Proposition}
     \newtheorem{defn}[thm]{Definition}
        
   \newtheorem{ques}[thm]{Question}

\theoremstyle{remark} 
\theoremstyle{remark} 
\theoremstyle{remark} 

    \numberwithin{equation}{section}

    \newcommand{\Neron}{N\'{e}ron~}

    \numberwithin{equation}{section}

\DeclareFontFamily{U}{wncy}{}
\DeclareFontShape{U}{wncy}{m}{n}{<->wncyr10}{}
\DeclareSymbolFont{mcy}{U}{wncy}{m}{n}

\begin{document}

\title[rank 1 and nontrivial $2$-part of Shafarevich-Tate groups over $\BQ_\infty$]{Elliptic curves with rank one and nontrivial $2$-part of Tate Shafarevich groups over the $\BZ_2$-extension of $\BQ$}

\author[Li-Tong Deng]{Li-Tong Deng}
\address{\textit{Li-Tong Deng}, Qiuzhen College, Tsinghua University, 100084, Beijing, China}
\email{\it dlt23@mails.tsinghua.edu.cn}

\author[Yong-Xiong Li]{Yong-Xiong Li}
\address{\textit{Yong-Xiong Li}, Yanqi Lake Beijing Institute of Mathematical Sciences and Applications, No. 544, Hefangkou Village Huaibei Town, Huairou District Beijing 101408, Beijing, China}
\email{\it yongxiongli@gmail.com}

\begin{abstract}
Let $\mathbb{Q}_\infty$ be the cyclotomic $\mathbb{Z}_2$-extension over $\mathbb{Q}$. For each integer $n\geq1$, let $\mathbb{Q}_n$ denote the unique subfield in $\mathbb{Q}_\infty$ such that $[\mathbb{Q}_\infty:\mathbb{Q}]=2^n$. Denote by $\mathbb{Z}_2[{\rm Gal}(\mathbb{Q}_n/\mathbb{Q})]$ the group ring of ${\rm Gal}(\mathbb{Q}_\infty/\mathbb{Q})$. For any elliptic curve defined over $\mathbb{Q}$ with odd conductor, the Mazur-Tate modular element associated with the curve is an element of $\mathbb{Z}_2[{\rm Gal}(\mathbb{Q}_n/\mathbb{Q})]$.    
In this paper, for each $n$, we study the $2$-adic properties of Mazur-Tate modular elements associated with quadratic twists of elliptic curves, under specializations by finite order characters of ${\rm Gal}(\mathbb{Q}_n/\mathbb{Q})$. Using the congruence properties of Heegner points and an equivariant version of the Coates-Wiles theorem, we construct an elliptic curve $E/\mathbb{Q}$ and a family of quadratic twists $E^{(m)}$ of $E$ such that each $E^{(m)}$ has both analytic and algebraic rank one over $\mathbb{Q}_\infty$, and whose Tate-Shafarevich group is infinite over $\mathbb{Q}_\infty$.  
\end{abstract}



\maketitle

\selectlanguage{english}

\section{Introduction}

\subsection{Background}
Let $p$ be a prime.
One motivation of this paper is the following question: 

\begin{ques}
Is there an elliptic curve $E$ defined over $\BQ$, and
are there infinitely many number fields $\fF_n/\BQ$ with $[\fF_n:\BQ]\to \infty$, 
such that $E(\fF_n)$ has rank $r_n=0$ (resp.  $r_n=1$) for all $n$, 
and the $p$-primary part $\Sha(E/\fF_n)(p)$ of $\Sha(E/\fF_n)$ is nontrivial (or trivial)? 
\end{ques}

For the nontrivial case, in \cite{kurihara}, Kurihara showed that there exists $E/\BQ$ and a prime $p$ that is a good supersingular odd prime for $E$, 
such that in the case $r_n=0$, the answer to the above question is affirmative when \(\mathfrak{F}_n\) is the \(n^{th}\) cyclotomic \(p\)-extension of \(\mathbb{Q}\). 
Later,  in \cite{KO2}, Kurihara and Otsuki extended this result to the case $r_n=1$ and $p=2$, 
assuming the finiteness of Tate-Shafarevich group for $E/\fF_1$, and they showed that the answer remains true under these conditions.
  
Also in the nontrivial case, when $r_n=0$ and $p=2$,  \cite[Theorem 8.1]{CL} shows that if $E$ is the Gross curve defined over a certain number field, 
the answer is also affirmative. For the trivial case, when $r_n=1$ and $p=2$,  
\cite[Proposition 5.2]{Li-aa} (see also \cite[Theorems 1.2 and 1.3]{CLL}) shows that there exists an elliptic curve $E/\BQ$ such that there exist number fields $\fF_n/\BQ$ with degree tending to 
infinity, where $E(\fF_n)$ has rank one, but $\Sha(E/\fF_n)(2)=0$ for all $n$.  

The first goal of this paper is to answer the above question for $r_n=1$ and $p=2$, without assuming the Birch-Swinnerton-Dyer conjecture or the finiteness of Tate-Shafarevich groups. 

\medskip

On the other hand, the study of twists of elliptic curves over $\BQ$ has a long history. For example, the study of quadratic twists of the elliptic curve
\[\CE: y^2=x^3-x\]
is related to the congruence number problem. Specifically, for a squarefree integer $d$, we define the quadratic twist $\CE^{(d)}$ of $\CE$ by the 
equation
\[\CE^{(d)}: dy^2=x^3-x.\]
Similarly, one can define the quadratic twist $E^{(d)}$ for any elliptic curve $E/\BQ$.
Another motivation for this paper is the following question: 

\begin{ques}
Is there an elliptic curve $E/\BQ$, and are there infinitely many $d$'s
such that $E^{(d)}(\BQ)$ has rank $0$ (resp. rank $1$), and $\Sha(E^{(d)}/\BQ)(p)\neq0$ (or \(=0\))? 
\end{ques}

For the nontrivial case, in \cite{razar}, for the elliptic curve $\CE/\BQ$, Razar showed there are infinitely many primes $\ell$ such that $\CE^{(\ell)}(\BQ)$ has rank zero and 
$\Sha(\CE^{(\ell)}/\BQ)(2)\neq0$. In \cite{KL-doc}, the second author and his collaborator provided 
evidence for the existence of infinitely many cubic twists of the curve $X_0(27)$ with rank one and nontrivial $2$-part of the Tate-Shafarevich groups. 

For the trivial case, Tian showed in \cite{tian} that there are infinitely many $d$, with arbitrarily finitely many prime factors, 
such that $\CE^{(d)}(\BQ)$ has rank 
one, but all have $\Sha(\CE^{(d)}/\BQ)(2)=0$. If further \(E\) is an elliptic curve over \(\mathbb{Q}\) with complex multiplication, then for almost all prime \(p\), by \cite{KO-inv} there are infinitely many \(d\) such that \(\Sha(E^{(d)}/\BQ)(p)=0\). 

\medskip

Building on our first goal, the main aim of this paper is to find a family of quadratic twists $E^{(d)}$ of an elliptic curve $E$ over $\BQ$, along with 
the number fields $\fF_n$ as described in the first question, 
such that ${\rm rank}(E^{(d)}(\fF_n))=1$ and $\Sha(E^{(d)}/\fF_n)(2)\neq0$. Furthermore, we
demonstrate the rank part of the Birch-Swinnerton-Dyer conjecture for all these elliptic curves $E^{(d)}/\fF_n$ as $n\to+\infty$. Our result generalizes previous ones from a fixed curve or a fixed number field \(\mathbb{Q}\) to an infinite family of number fields whose degree tends to infinity, and to a quadratic twist family of \(E\) over these fields.

\bigskip

\subsection{Main Results}
Let $E$ be the elliptic curve defined by
\[E: y^2+y=x^3+2.\]
The curve $E$ has conductor $243$, and its Mordell-Weil group $E(\BQ)$ is isomorphic to $\BZ/3\BZ$. Additionally, the algebraic value 
of the $L$-function of $E$ at $s=1$ is given by
\[\frac{L(E,1)}{\Omega_E}=\frac{1}{3},\] 
where $\Omega_E$ is the real period of $E$. Let $f=\sum^\infty_{n=1}a_nq^n$ be the normalized weight $2$ newform corresponding to $E$, 
where $q=e^{2\pi i z}$. Note that $a_2=0$, and $E$ has good supersingular reduction at $2$. 
For a squarefree integer $d\equiv 1\mod 4$, we denote by $E^{(d)}$ the quadratic twist of $E$ under the extension $\BQ(\sqrt{d})/\BQ$. 
\medskip

Let $\BQ_\infty$ be the cyclotomic $\BZ_2$-extension of $\BQ$. For each integer $n\geq0$, we denote by $\BQ_n$ the unique intermediate field 
in $\BQ_\infty$ such that $[\BQ_n:\BQ]=2^n$. Note that $\BQ=\BQ_0$. 

\medskip

For each $n\geq0$, the $2$-power Selmer group $\Sel(E^{(d)}/\BQ_n)$ fits into the middle of the following exact sequence:
\[0\to E^{(d)}(\BQ_n)\otimes (\BQ_2/\BZ_2)\to \Sel(E^{(d)}/\BQ_n)\to \Sha(E^{(d)}/\BQ_n)(2)\to 0.\]
We denote by $\Sel(E^{(d)}/\BQ_n)^\vee$ the Pontryagin dual of $\Sel(E^{(d)}/\BQ_n)$. The ($\BZ_2$-)corank of $\Sel(E^{(d)}/\BQ_n)$ is defined to be
the $\BZ_2$-rank of its Pontryagin dual. Let $L(E^{(d)}/\BQ_n,s)$ denote the $L$-series of $E^{(d)}$ over $\BQ_n$. From Deuring's theorem, this is 
a holomorphic function on the whole complex plane.

\medskip

For any squarefree integer $\ft$,
we denote by $h(\ft)$ the class number of the quadratic field $\BQ(\sqrt{\ft})$. 
For an odd prime $\ell$, we denote by $\left(\frac{\cdot}{\ell}\right)$ the Legendre symbol. 
For a finite abelian group $\fA$, we define the $2$-rank of $\fA$ to be the $\BF_2$-dimension of $\fA/2\fA$. The main theorem of our paper is as follows.

\medskip

\begin{thm}\label{thm2}
Suppose that $p$ and $q$ are distinct primes satisfying the following conditions:
\begin{enumerate}
\item[(1)]$p\equiv q\equiv 7\mod 12$;
\item[(2)]$a_p$ and $a_q$ are odd;
\item[(3)]$\left(\frac{-2p}{q}\right)=1$
\item[(4)]$3\nmid h(-q)h(-6pq)$.
\end{enumerate}
Let $m=pq$. Then the following hold:
\begin{enumerate}
\item[(i)]${\rm rank}(E^{(m)}(\BQ))=0=\ord_{s=1}L(E^{(m)}/\BQ,s)$.
\item[(ii)]For each integer $n\geq1$, we have
$$\ord_{s=1}L(E^{(m)}/\BQ_n,s)=1={\rm rank}(E^{(m)}(\BQ_n)).$$
Furthermore, let
\[\Sha(E^{(m)}/\BQ_\infty)=\lim_n\Sha(E^{(m)}/\BQ_n).\]
Then both the order and the $2$-rank of $\Sha(E^{(m)}/\BQ_\infty)(2)$ are infinite.
\end{enumerate}
\end{thm}

We remark that there are infinitely many primes $p$ and $q$ satisfying conditions (1) and (2) by the Chebotarev density theorem. 
There are many examples of primes $p$ and $q$ satisfying the conditions (1)--(4) in Theorem \ref{thm2}. For instance, one can take $p=31$ and $q=7$.  
According to the Cohen-Lenstra heuristic, there should be infinitely many such pairs of primes.  For further examples of primes $p$ and $q$ satisfying the conditions of Theorem \ref{thm2}, we refer the reader to the numerical data in Appendix C.

\bigskip

\subsection{Idea for the Proofs}

The key ingredients for proving Theorems \ref{thm2} are as follows;
\begin{itemize}
\item The Mazur-Tate elements introduced in \cite{MT-d1}, their vertical distribution properties, integral properties of the modular symbols (\cite{manin}) and especially the congruence relations between every cyclotomic specialization under quadratic twists;
\item  Iwasawa theory for elliptic curves with complex multiplication for odd good ordinary primes, as developed by Coates, Wiles and Rubin (\cite{coates1}, \cite{rubin});
\item  The congruence properties of the Heegner points (\cite{KL-sigma}), derived from the $p$-adic Waldspurger formula developed by Bertolini, Darmon and Prasanna \cite{BDP}.
\end{itemize}
Indeed, using the first ingredient, we can show that 
$$L(E^{(m)},\eta,1)\neq0$$ 
for all finite order characters $\eta\neq\psi_1$ of ${\rm Gal}(\BQ_\infty/\BQ)$. Here $L(E^{(m)},\eta,s)$ denotes the $L$-series of $E^{(m)}/\BQ$ twisted by $\eta$, and $\psi_1$ is the nontrivial character of ${\rm Gal}(\BQ_1/\BQ)$.
Using the third ingredient, we can show that the Heegner point associated to $E^{(-3q)}$ and $\BQ(\sqrt{-2p})$ is non-torsion. Combining this with the Gross-Zagier theorem and Kolyvagin theorem,  we obtain 
$${\rm ord}_{s=1}L(E^{(m)},\psi_1,s)=1.$$
Thus, the statement on the analytic rank in Theorem \ref{thm2} follows. The algebraic rank statement follows from an equivariant Coates-Wiles theorem for $E^{(m)}/\BQ_n(\sqrt{-3})$, which involves the second ingredient. Theorem \ref{thm2} then follows from the standard fact that the $2$-power Selmer group of $E^{(m)}$ over $\BQ_\infty$
is free of rank one over the Iwasawa algebra (see \cite{CS}), together with the observation that a finitely generated $\BZ_2$-module must be of torsion over the Iwasawa algebra. 

\medskip

For each finite-order character $\rho$ of ${\rm Gal}(\BQ_\infty/\BQ)$,
determining the $2$-adic properties of the Mazur-Tate elements at $\rho$ presents a challenge, 
as the original method due to Kurihara \cite{kurihara} (later modified by Pollack \cite{pollack}) is not directly applied in our setting. 
Indeed, while the method can determine the $2$-adic properties of Mazur-Tate elements associated with $E$ at $\rho$, it fails for 
the Mazur-Tate elements associated with $E^{(m)}$. 
To overcome this difficulty, we define a modular element $\xi_{E}(\chi)$ associated 
with the pair $(E,\chi)$, where \(\chi\) is the nontrivial character of \(\mathbb{Q}(\sqrt{m})/\mathbb{Q}\). The specializations of $\xi_{E}(\chi)$ at $\rho$ are equal to the specializations of the Mazur-Tate elements associated 
with $E^{(m)}$ up to $2$-adic units. 
By employing the integral properties of modular symbols (see \cite{manin}), we establish that the specializations of $\xi_E(\chi)$ 
and the Mazur-Tate element associated with $E$ are related modulo $2$, and their differences are a product of $2$-adic units coming from Euler factors. 
This key observation allows us to resolve the problem, except in the case of the primitive characters of ${\rm Gal}(\BQ_3/\BQ)$.  
To handle this exceptional case, we consider one such character $\rho_3$. 
Using the Hecke action on modular symbols, we derive several higher congruence  
relations (i.e. modulo $4$) between the coefficients in the  
specialization of  $\xi_E(\chi)$ at $\rho_3$. 
By combining these higher congruence relations with the $2$-adic property of $\rho_3(\xi_E(\chi))$, which is obtained from the modulo $2$ congruence, we can establish the required $2$-adic property for $\rho_3(\xi_E(\chi))$. Equivalently, we get the $2$-adic properties for 
the Mazur-Tate elements associated with $E^{(m)}$ at $\rho_3$.

\medskip

We remark that the method described above has several applications to other arithmetic problems. In \cite{DL}, we determine the Iwasawa invariants for the higher even $K$-groups of the rings of integers in the $\BZ_2$-extension of real quadratic number fields. In \cite{D},  the first author determines the variations of the analytic Iwasawa invariants of elliptic curves over the $\BZ_2$-extension of $\BQ$ under quadratic twists.

\medskip

\subsection{Layout of the Paper} 
In the final part of the introduction, we outline the structure of the paper.  In Section \ref{sec2}, we introduce some notation and provide the necessary preliminaries. In Section \ref{sec3}, we discuss the Mazur-Tate elements and their relationship with the special values of $L$-functions associated with elliptic curves.  Section \ref{sec4} is dedicated to determining the $2$-adic properties of Mazur-Tate elements for $E^{(m)}$ at all finite-order characters of ${\rm Gal}(\BQ_\infty/\BQ)$.  
Section \ref{sec5} utilizes the congruence properties of Heegner points (as established by Kriz and Li \cite{KL-sigma}) to demonstrate the non-torsionness of these points, and proves Theorem \ref{thm2}. In Appendix A , we provide a detailed analysis of the $2$-adic properties of the Mazur-Tate elements for $E^{(m)}$ at the primitive characters of ${\rm Gal}(\BQ_3/\BQ)$.  In Appendix B, we give a sketch proof of an equivariant Coates-Wiles theorem. Finally, in Appendix C, we present numerical data for primes $p$ and $q$ satisfying the conditions of Theorem \ref{thm2}.

\medskip

\noindent
{\bfseries Acknowledgement} 
The authors would like to thank Ye Tian and Shuai Zhai for helpful discussions related to the paper.

\section{Notation and Preliminaries}\label{sec2}

\subsection{General Setting} Let $E$ denote the elliptic curve given by
\[E:\quad y^2+y=x^3+2\]
which is defined over $\BQ$ and is a sextic twist of $X_0(27)$. The elliptic curve $E$ has complex multiplication by $\BQ(\sqrt{-3})$, conductor $N=N_E=243$, discriminant $\Delta_E=-2187$, and real period
\[\Omega_E=3.6747....\]
Furthermore, the Mordell-Weil group of $E/\BQ$, denoted by $E(\BQ)$, is isomorphic to $\BZ/3\BZ$.

Let 
\[f=\sum^\infty_{n=1}a_nq^n,\qquad q=\exp(2\pi iz),\]
be the weight $2$ newform associated to $E$. The $L$-function of $E$ is denoted by $L(E,s)$, and we have 
\[L(E,s)=L(f,s).\] 
It is known that
\[\frac{L(E,1)}{\Omega_E}=\frac{1}{3}.\] 
Let $\epsilon_E$ denote the root number of $L(E,s)$. For any Dirichlet character $\eta$, we denote by $L(E,\eta,s)$ the $L$-function of $E$ twisted by $\eta$.

For an abelian group $\fA$ and a positive integer $n$, we denote by $\fA_n$ the kernel of the multiplication by $n$ map on $\fA$. We set
\[\fA(n)=\bigcup_{m\geq1}\fA_{n^m}.\] 

For a squarefree integer $d\equiv 1\mod 4$, we denote by $E^{(d)}$ the elliptic curve obtained by twisting $E$ under the quadratic extension $\BQ(\sqrt{d})/\BQ$. For any number field $F$, we define the $2$-power Selmer group via the following exact sequence:
\[0\to\Sel(E^{(d)}/F)\to H^1(F,E^{(d)}(2))\to\prod_{v}\frac{H^1(F_v,E^{(d)}(2))}{E^{(d)}(F_v)\otimes(\BQ_2/\BZ_2)},\]
where the product is taken over all places of $F$. Denote by $\Sha(E^{(d)}/F)(2)$ the $2$-primary part of the Tate-Shafarevich group of $E^{(d)}/F$. We have an exact sequence:
\[0\to E^{(d)}(F)\otimes(\BQ_2/\BZ_2)\to \Sel(E^{(d)}/F)\to \Sha(E^{(d)}/F)(2)\to 0.\]
For any finite prime $\fl$ of $F$, we denote by $c_\fl(E^{(d)}/F)$ the Tamagawa number of $E^{(d)}/F$ at $\fl$.

\subsection{Cyclotomic Extensions}

For each non-negative integer $n$, let $\zeta_n$ denote a primitive $n$-th root of unity.  Let $\BQ_n$ denote the maximal totally real subfield in $\BQ(\zeta_{2^{n+2}})$. Denote by $\msg_n={\rm Gal}(\BQ(\zeta_{2^{n}})/\BQ)$ and $G_n={\rm Gal}(\BQ_n/\BQ)$. It is straightforward to observe that $G_n\simeq\BZ/2^n\BZ$.

The natural projection $\msg_n\to \msg_{n-1}$ induces a map
\[\pi_n:\BQ[\msg_n]\to \BQ[\msg_{n-1}].\]
Define 
\[\nu_{n-1}:\BQ[\msg_{n-1}]\to\BQ[\msg_n]\]
by sending any $\sigma\in\msg_{n-1}$ to 
\[\sum_{\pi_n(\sigma')=\sigma}\sigma',\] 
where the sum is over all $\sigma'\in\msg_n$ that map to $\sigma$ under the projection $\pi_n$.
Note that we also have canonical projection maps
\[\msg_{n+2}\to G_n.\]
Similarly, we can define the maps
\[\pi_n: \BQ[G_n]\to \BQ[G_{n-1}]\]
and
\[\nu_{n-1}: \BQ[G_{n-1}]\to \BQ[G_n].\]

Now define $\BQ_\infty=\bigcup_{n\geq0}\BQ_n$, which is called the cyclotomic $\BZ_2$-extension of $\BQ$. Let
\[\Gamma={\rm Gal}(\BQ_\infty/\BQ)\quad\textrm{and} \quad\Gamma_n={\rm Gal}(\BQ_\infty/\BQ_n).\]
We define the group algebras $\Lambda_n=\BZ_2[G_n]$ associated with $G_n$ and $\Lambda=\BZ_2[[\Gamma]]$, which is the Iwasawa algebra of $\Gamma$. 

For a $\BZ_2$-module $\fM$, we denote by $\fM^\vee={\rm Hom}(\fM,\BQ_2/\BZ_2)$ the Pontryagin dual of $\fM$. Let $\fM_{\rm tor}$ denote the torsion submodule of $\fM$.

In $\BQ_2(\zeta_{2^n})$, we define $v$ to be the normalized additive valuation at the unique prime above $2$, satisfying $v(2)=1$. 

\subsection{An Application of Chebotarev Density Theorem}

Let $p$ and $q$ be two distinct primes satisfying the following conditions:
\begin{itemize}
\item $p\equiv q\equiv 7\mod 12$,
\item $a_p$ and $a_q$ are odd.
\end{itemize}
Here, $a_p$ and $a_q$ are the Fourier coefficients of $f$.
We denote by $m=pq$ and $\chi$ the Dirichlet character associated with the quadratic extension $\BQ(\sqrt{m})/\BQ$. 

We define the set $\fS$ as follows:
\[\fS=\{p:\,  p \textrm{ is a prime, } p\equiv 7\mod 12, \quad a_p\not\equiv0\mod 2\}.\]

\begin{prop}
The set $\fS$ has infinite cardinality.
\end{prop}

\begin{proof}
Since $E$ has good reduction at $p$, let $\widetilde{E}$ denote the reduction of $E$ modulo $p$. We observe that
\[a_p=p+1-\#(\widetilde{E}(\BF_p)).\]
Therefore, the condition $2\nmid a_p$ is equivalent to $\widetilde{E}(\BF_p)_2=0$. A direct computation shows that 
\[L=\BQ(E_2)=\BQ(\sqrt[3]{18},\sqrt{-3}).\] 
For \(p\equiv 1 \mod 3\), since $p$ splits in $\BQ(\sqrt{-3})$, it follows that $2\nmid a_p$ if and only if $p$ is inert in $\BQ(\sqrt[3]{18})/\BQ$. Consider the following field extensions
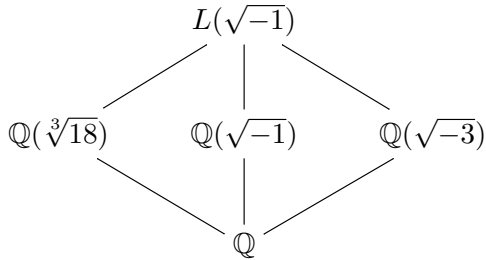
\[\xymatrix{
     & L(\sqrt{-1}) \ar@{-}[dr] \ar@{-}[dl]\ar@{-}[d] &\\
     \mathbb{Q}(\sqrt[3]{18}) \ar@{-}[dr] &\BQ(\sqrt{-1})\ar@{-}[d] & \mathbb{Q}(\sqrt{-3}) \ar@{-}[dl] \\
     & \mathbb{Q} &}.\]
The field extension $L(\sqrt{-1})/\BQ$ is Galois. By the Chebotarev density theorem, there exists infinitely many primes in $\BQ$ that are inert in both $\BQ(\sqrt{-1})/\BQ$ and $\BQ(\sqrt[3]{18})/\BQ$, but split in $\BQ(\sqrt{-3})/\BQ$. Therefore, the set $\fS$ is infinite.     

\end{proof}

\section{Mazur-Tate Elements and Special Value Formulas of $L$-Functions}\label{sec3}

In this section, we will introduce modular symbols and Mazur-Tate modular elements for a general elliptic curve $A$ defined over $\BQ$. We will then study the distribution properties of these modular elements and their relationship with special values of the $L$-functions of $A$. In the final part of this section, we will focus on the case when $A=E$ or a certain quadratic twist of $E$, and present results that allow us to determine the algebraic part of $L(A,1)$.

\subsection{Modular Symbols}
For a positive integer $N$, we denote by $S_2(\Gamma_0(N))$ the space of cusp forms of weight $2$ with respect to $\Gamma_0(N)$.  Let $\phi=\sum_{n\geq1}a_n(\phi)q^n$ be a normalized newform in $S_2(\Gamma_0(N))$, where we assume that $a_n(\phi)\in\BQ$ for all $n\geq1$. It is well-known that $\phi$ corresponds to an isogeny class of elliptic curves defined over $\BQ$, and we denote by $A$ the unique optimal elliptic curve in this $\BQ$-isogeny class. We also write $L(A,s)$ for the complex $L$-function associated with $A$. The $L$-function $L(A,s)$ coincides with $L(\phi,s)$, and there exists a non-constant morphism defined over $\BQ$,
\[\iota: X_0(N)\to A,\]  
which does not factor through any other curves in the isogeny class of $A$. Let $\omega$ denote the \Neron differential on a global Weierstrass minimal model of $A$. Then there exists a nonzero integer $\nu_A$ such that 
\begin{equation}\label{3-1-f}
\iota^*\omega=\nu_A\cdot \phi(z)dz. 
\end{equation}
The number $\nu_A$ is called the Manin constant, and it is known that $\nu_A$ is odd when $N$ is odd \cite{manin-c}. 

Let $\BH$ be the upper half plane, and let $\BH^*=\BH\cup\BP^1(\BQ)$. Let $g$ be an element of $\Gamma_0(N)$. Let $\alpha$ and $\beta$ be two points in $\BH^*$ such that $\beta=g\alpha$. Then, any path from $\alpha$ to $\beta$ in $\BH^*$ is a closed path on
$X_0(N)$, whose homology class depends only on $\alpha$ and $\beta$. Thus, it determines an integral homology class in $H_1(X_0(N),\BZ)$, which we denote by the modular symbol $\{\alpha,\beta\}$.  For further details on the basic properties of these modular symbols, one may refer to \cite{manin}, \cite{cremona} and \cite{zhai1}.

From now on, we assume that the conductor $N$ of $A$ is odd and that the discriminant $\Delta_A$ of $A$ is negative. 
For a modular symbol $\{\alpha,\beta\}\in H_1(X_0(N),\BZ)$, we define the pairing
\[\langle \{\alpha,\beta\},\phi\rangle=2\pi i\int^\beta_\alpha \phi(z)dz. \]
The period lattice $\Lambda_\phi$ of the modular form $\phi$ is defined as the set of all such pairings $\langle \{\alpha,\beta\},\phi\rangle$ for all modular symbols 
$\{\alpha,\beta\}\in H_1(X_0(N),\BZ)$. Since $\Delta_A<0$, it is known (see \cite{MT-d1}) that there exist two nonzero real numbers $\Omega^{\pm}_\phi$ such that
\begin{equation}\label{3-2-f}
\Lambda_\phi=\BZ\cdot\Omega^+_\phi+\BZ\cdot\frac{\Omega^+_\phi+i\Omega^-_\phi}{2}.  
\end{equation}
Let $\Lambda_A$ denote the period lattice of the \Neron differential $\omega$ for $A$. From \eqref{3-1-f}, we have
\[\Lambda_A=\nu_A\Lambda_\phi.\]
We define $\Omega^+_A$ (resp. $i\Omega^-_A$) as the minimal positive real (resp. purely imaginary) period of $A$. Thus, we have
\[\Omega^+_A=\nu_A\Omega^+_\phi\quad\textrm{ and }\quad \Omega^-_A=\nu_A\Omega^-_\phi.\]
Since we assume that $A$ has odd conductor, $\nu_A$ is odd. For the purpose of $2$-adic valuation, we may assume that
$\nu_A=1$. In fact, when $A=E$, where $E$ is optimal in its isogeny class, it is always the case that $\nu_E=1$. For simplicity, we denote $\Omega_A=\Omega^+_A$.   

\medskip

Let $\frac{k}{t}$ be a rational number such that $(k,t)=1$ and $(t,N)=1$. From \cite[Proposition 2.2 and Theorem 3.9]{manin}, we have
\[\{0,\frac{k}{t}\}\in H_1(X_0(N),\BZ).\]
We then define
\[\CS\left(\frac{k}{t}\right)=\Re\left(\langle\{0,\frac{k}{t}\},\phi\rangle/\Omega_A\right).\]
From equation \eqref{3-2-f}, we have $\CS\left(\frac{k}{t}\right)\in \frac{1}{2}\BZ$.

\begin{prop}\label{prop3-1}
\noindent
\begin{enumerate}
\item[(1)]We have $\CS(0)=0$. For any rational number $\beta=\frac{k}{t}$ satisfying $(k,t)=1$ and $(t,N)=1$, we have $\CS(\beta)=\CS(-\beta)$, and $\CS$ is periodic with period $1$, i.e., $\CS(\beta+1)=\CS(\beta)$.
\item[(2)]Let $\ell$ be a prime such that $(\ell,N)=1$. Let $\beta$ be as in (1), then we have
\begin{equation}\label{3-3-f}
\sum_{x\in\BZ/\ell\BZ}\CS\left(\frac{\beta+x}{\ell}\right)=a_\ell(\phi)\cdot \CS(\beta)-\CS(\ell\beta)+\sum_{x\in\BZ/\ell\BZ}\CS\left(\frac{x}{\ell}\right).
\end{equation}
\end{enumerate}
\end{prop}

\begin{proof}
(1) follows directly from the definition of the modular symbols. For example, the equation $\CS(\beta)=\CS(-\beta)$ follows from the fact that the complex conjugation of $\langle \{0 ,\beta\},\phi\rangle$ is equal to $\langle \{0 ,-\beta\},\phi\rangle$ (see \cite{cremona}). Let $\BT_\ell$ be the Hecke operator associated with $\Gamma_0(N)$. From the Hecke action on modular symbols (see \cite{cremona}), we have
\[\BT_\ell\{\alpha,\gamma\}=\{\ell\alpha,\ell\gamma\}+\sum_{x\in\BZ/\ell\BZ}\left\{\frac{\alpha+x}{\ell},\frac{\gamma+x}{\ell}\right\}.\]
Since the pairing is compatible under the Hecke action, i.e. 
$$\langle \BT_\ell\{\alpha,\gamma\},\phi\rangle=\langle \{\alpha,\gamma\},\BT_\ell\phi\rangle,$$ 
and $\BT_\ell\phi=a_\ell(\phi)\phi$, (2) follows by taking $\alpha=0$ and $\gamma=\beta$.
\end{proof}

\subsection{Modular Elements}
For a positive integer $n$, we recall that $\msg_n={\rm Gal}(\BQ(\zeta_{2^n})/\BQ)$, which is isomorphic to $(\BZ/2^n\BZ)^\times$. For each $k\in(\BZ/2^n\BZ)^\times$, we denote by $\sigma_k$ the element in $\msg_n$ corresponding to the automorphism of $\BQ(\zeta_{2^n})$ that sends $\zeta_{2^n}$ to $\zeta^k_{2^n}$.

The modular element associated with $A$ for $\msg_n$ is defined as follows:
\[\xi_{2^n}=\sum_{k\in(\BZ/2^n\BZ)^\times}\CS\left(\frac{k}{2^n}\right)\sigma_k\in \BQ[\msg_n],\]
and the trace element $\tau_{2^n}$ in $\BQ[\msg_n]$ is given by
\[\tau_{2^n}=\sum_{\sigma\in\msg_n}\sigma.\]

\begin{lem}\label{lem3-2-1}
Let $\phi$ be the weight $2$ newform associated with $A$,
and assume that $a_2(\phi)=0$.
For any integer $n\geq3$, we have the following distribution relation
\[\pi_n(\xi_{2^n})=-\nu_{n-2}(\xi_{2^{n-2}})+\CS\left(\frac{1}{2}\right)\tau_{2^{n-1}}.\]
\end{lem}

\begin{proof}
By the definition of $\xi_{2^n}$, we have
\[\pi_n(\xi_{2^n})=\pi_n\left(\sum_{1\leq k\leq 2^{n-1}, 2\nmid k}\left(\CS\left(\frac{k}{2^n}\right)\sigma_k+\CS\left(\frac{k+2^{n-1}}{2^n}\right)\sigma_{k+2^{n-1}}\right)\right).\]
Using the fact that 
\[\pi_n(\sigma_{k+2^{n-1}})=\pi_n(\sigma_k),\]
we obtain
\[\pi_n(\xi_{2^n})=\sum_{1\leq k\leq 2^{n-1},2\nmid k}\left(\CS\left(\frac{k}{2^n}\right)+\CS\left(\frac{k+2^{n-1}}{2^n}\right)\right)\sigma_k.\]
Since the conductor $N$ of $A$ is odd and $a_2(\phi)=0$, we take $\ell=2$ in (2) of Proposition \ref{prop3-1}, which gives
\[\CS\left(\frac{k}{2^n}\right)+\CS\left(\frac{k+2^{n-1}}{2^n}\right)=-\CS\left(\frac{k}{2^{n-2}}\right)+\CS\left(\frac{1}{2}\right).\]
Combining this with the relation 
\[\nu_{n-2}(\sigma_k)=\sigma_k+\sigma_{k+2^{n-2}},\]
the lemma follows.
\end{proof}

\medskip

We define the Mazur-Tate element $\xi_{\BQ_n}$ as the image of $\xi_{2^{n+2}}$ under the natural restriction map
\[\BQ[\msg_{n+2}]\to\BQ[G_n].\]
We also denote by
\[\tau_{\BQ_n}=\sum_{\sigma\in G_n}\sigma\] 
the trace element of $G_n$. From Lemma \ref{lem3-2-1} and the fact that $\nu_A$ is odd, we can easily prove the following lemma.

\begin{lem}\label{lem3-2-2}
\noindent
\begin{enumerate}
\item[(1)]For any integer $n\geq 2$, we have
\[\pi_n(\xi_{\BQ_n})=-\nu_{n-2}(\xi_{\BQ_{n-2}})+2\CS\left(\frac{1}{2}\right)\tau_{\BQ_{n-1}}.\]
\item[(2)]The modular element $\xi_{\BQ_n}$ belongs to  $\BZ_2[G_n]$.
\end{enumerate}
\end{lem}

\medskip

\subsection{Relation with $L$-functions}

Let $M$ be an integer such that $(M,N)=1$.  Let $\eta$ be a primitive Dirichlet character modulo $M$. Define the Gauss sum of $\eta$ by
\[g(\eta)=\sum^M_{x=1}\eta(x)\exp\left(\frac{2\pi ix}{M}\right).\]
We record the special value formulas in the following theorem (see \cite[Theorem 4.2]{manin}).

\begin{thm}\label{thm3-3-1}
\noindent
\begin{enumerate}
\item[(1)]Let $\sigma(M)$ denote the sum of divisors of $M$, i.e., $\sigma(M)=\sum_{l\mid M}l$. Then we have
\[\left(a_M(\phi)-\sigma(M)\right)L(A,1)=\sum_{l\mid M}\sum_{x\in\BZ/l\BZ}\langle \{0,\frac{x}{l}\},\phi \rangle.\]
\item[(2)]Assume that $\eta$ is a nontrivial primitive even Dirichlet character modulo $M$. Then we have
\[g(\eta)\frac{L(A,\overline{\eta},1)}{\Omega_A}=\sum_{x\in(\BZ/M\BZ)^\times}\eta(x)\cdot\CS\left(\frac{x}{M}\right). \]
\end{enumerate}
\end{thm}

Inspiring by the second part of Theorem \ref{thm3-3-1}, for any primitive Dirichlet character $\psi$ with conductor dividing $M$, we define
\begin{equation}\label{3-4-f}
\msp_{\psi,M}(\CS)=\sum_{x\in(\BZ/M\BZ)^\times}\psi(x)\CS\left(\frac{x}{M}\right).
\end{equation}
When $\psi$ has conductor exactly $M$, we simply denote $\msp_{\psi,M}(\CS)$ by $\msp_\psi(\CS)$. 

\medskip

\begin{lem}\label{lem3-3-2}
Let $M$ be a squarefree positive integer such that $(M,N)=1$.
\begin{enumerate}
\item[(1)]Let $\mathds{1}$ be the trivial character, then
\[\msp_{\mathds{1},M}(\CS)=\left(\prod_{p\mid M}(a_p(\phi)-2)-\varphi(M)\right)\cdot\frac{L(A,1)}{\Omega_A},\]
where the product is taken over all prime divisors $p$ of $M$, and $\varphi(M)$ is the Euler function.
\item[(2)]For any prime $\ell\mid M$, let $\psi$ be a nontrivial primitive Dirichlet character with conductor dividing $M/\ell$. Then
\[\msp_{\psi,M}(\CS)=\left(a_\ell(\phi)-\psi(\ell)-\overline{\psi}(\ell)\right)\msp_{\psi,M/\ell}(\CS).\]
\end{enumerate}
\end{lem}
We remark that for $\psi$ a quadratic character, an analogue of (2) can be found in \cite[(2.5)]{zhai2} and the references cited there.

\begin{proof}
(1) Observe the identity
\[\sum_{k\in\BZ/M\BZ}\CS\left(\frac{k}{M}\right)=\sum_{d\mid M}\msp_{\mathds{1},d}(\CS).\]
Viewing the left hand side as a function of $M$, and applying the Möbius inversion formula, we obtain
\begin{equation}\label{3-3-f1}
\msp_{\mathds{1},M}(\CS)=\sum_{d\mid M}\mu(M/d)\sum_{k\in\BZ/d\BZ}\CS\left(\frac{k}{d}\right).
\end{equation}
Note that $(a_d(\phi)-\sigma(d))\frac{L(A,1)}{\Omega_A}$ is real. By (1) of Theorem \ref{thm3-3-1}, we have
\[\sum_{l\mid d}\sum_{k\in \BZ/l\BZ}\CS\left(\frac{k}{l}\right)=(a_d(\phi)-\sigma(d))\frac{L(A,1)}{\Omega_A}.\]
Therefore, by applying the Möbius inversion formula, we have
\[\sum_{k\in\BZ/d\BZ}\CS\left(\frac{k}{d}\right)=\sum_{l\mid d}\mu(d/l)\left(a_l(\phi)-\sigma(l)\right)\frac{L(A,1)}{\Omega_A}.\]
Since $M$ is squarefree, $d$ is also squarefree.
Using the identities
\[\sum_{l\mid d}\mu(d/l)a_l(\phi)=\prod_{p\mid d}(a_p(\phi)-1),\]
where, in the product, $p$ runs over all prime divisors of $d$, and
\[\sum_{l\mid d}\mu(d/l)\sigma(l)=d,\]
the equation \eqref{3-3-f1} becomes 
\[\msp_{\mathds{1},M}(\CS)=\sum_{d\mid M}\mu(M/d)\cdot\left(\prod_{p\mid d}(a_p(\phi)-1)-d\right)\frac{L(A,1)}{\Omega_A}.\]
Noting that $\sum_{d\mid M}\mu(M/d)d=\varphi(M)$, and
\[\sum_{d\mid M}\mu(M/d)\prod_{p\mid d}(a_p(\phi)-1)=\prod_{p\mid M}(a_p(\phi)-2),\]
we can conclude that part (1) follows.

(2) Denote $M'=M/\ell$. Since $(\ell,M')=1$, any element in $(\BZ/M\BZ)^\times$ can be written uniquely as
$xM'+y\ell$, where $x\in(\BZ/\ell\BZ)^\times$ and $y\in(\BZ/M'\BZ)^\times$. 
 A direct computation shows that
\begin{equation}\label{3-3-f2}
\msp_{\psi,M}(\CS)=\sum_{y\in(\BZ/M'\BZ)^\times}\sum_{x\in\BZ/\ell\BZ}\psi(y)\CS\left(\frac{y+M'x}{M}\right)-
\sum_{y\in(\BZ/M'\BZ)^\times}\psi(\ell y)\CS\left(\frac{\ell y}{M}\right).\end{equation}
Taking $\beta=\frac{y}{M'}$ in \eqref{3-3-f} and noting that $\psi$ is a nontrivial character, from \eqref{3-3-f2}, we obtain
\[\begin{aligned}\msp_{\psi,M}(\CS)=&a_\ell(\phi)\sum_{\ell\in(\BZ/M'\BZ)^\times}\psi(y)\CS\left(\frac{y}{M'}\right)-
\sum_{y\in(\BZ/M'\BZ)^\times}\psi(y)\CS\left(\frac{y\ell}{M'}\right)\\
-&\sum_{\ell\in(\BZ/M'\BZ)^\times}\psi(\ell y)\CS\left(\frac{y}{M'}\right).
\end{aligned}\]  
Since $\ell$ is prime to $M'$, the second sum is equal to $\overline{\psi}(\ell)\msp_{\psi,M'}(\CS)$. Thus, part (2) follows.
\end{proof}

\begin{lem}\label{lem3-3-3}
Let $\ell$ be a prime and $M=\ell^n$. Assume that $(M,N)=1$. Then, we have
\[\msp_{\mathds{1},M}(\CS)=\left((a_{M}-2a_{M/\ell}+a_{M/\ell^2})-\varphi(M)\right)\cdot\frac{L(A,1)}{\Omega_A}.\]
Here, we have used the convention $a_l=0$ if $l\notin\BZ$.
\end{lem}
\begin{proof}
Consider the following identity:
\[\sum_{d\mid M}\mu(M/d)a_d=a_M-a_{M/\ell},\]
where the sum is taken over all positive divisors $d$ of $M$. This follows from
the identity
\[a_{M}=\sum_{\ell^i\mid M}(a_{\ell^i}-a_{\ell^{i-1}})\]
and Möbius inversion formula. Next, 
by applying the same method, we obtain 
$$\sum_{d\mid M}\mu(M/d)a_{d/\ell}=a_{M/\ell}-a_{M/\ell^2}.$$
Thus, the lemma follows from a similar argument as in (1) of Lemma \ref{lem3-3-2}.
\end{proof}

\medskip

\subsection{$L$-Values for Quadratic Twists of $E$}
In the last subsection, we take $A$ to be the elliptic curve
\[E: y^2+y=x^3+2.\]
Note that the conductor $N=243$, \(a_2=0\) and $\Delta_E<0$, our assumption is automatically satisfied. We simply write $a_n$ for $a_n(\phi)$ in the follows.

\begin{lem}\label{s(k/8)}
\noindent
\begin{enumerate}
\item[(1)]$\CS\left(\frac{1}{2}\right)=-1$ and $\CS\left(\frac{1}{4}\right)=\CS\left(\frac{3}{4}\right)=-\frac{1}{2}$.
\item[(2)]$\CS\left(\frac{x}{8}\right)=0$ for all odd $x$.
\end{enumerate}
\end{lem}

\begin{proof}
The first assertion of part (1) follows directly from Lemmas \ref{lem3-3-2} and \ref{lem3-3-3}. We now prove the second assertion.  
Since $\CS\left(\frac{1}{4}\right)=\CS\left(\frac{3}{4}\right)$, by applying Lemma \ref{lem3-3-3}, we obtain
\[2\CS\left(\frac{1}{4}\right)=(a_4-2a_2-1)\cdot\frac{L(E,1)}{\Omega_E}.\]
Recall that $L(E,1)/\Omega_E=1/3$, and observe that $a_4=-2$ and $a_2=0$, we have
$\CS\left(\frac{1}{4}\right)=-1/2$. 

For part (2), note that 
\[\CS\left(\frac{1}{8}\right)=\CS\left(\frac{7}{8}\right)\quad\textrm{and}\quad \CS\left(\frac{3}{8}\right)=\CS\left(\frac{5}{8}\right).\]
Since $a_8=0$, and applying Lemma \ref{lem3-3-3}, we have $\CS\left(\frac{1}{8}\right)=-\CS\left(\frac{3}{8}\right)$. Next, we consider the elliptic curve
$E^{(2)}$, which is the quadratic twist of $E$ by the extension $\BQ(\sqrt{2})/\BQ$. Let $\eta_8$ be the Dirichlet character associated to $\BQ(\sqrt{2})/\BQ$. Since $E^{(2)}$ has root number $-1$, applying (2) in Theorem \ref{thm3-3-1} to $E$ and $\eta_8$, we obtain
\[\sum^8_{x=1}\eta_8(x)\CS\left(\frac{x}{8}\right)=0,\]
which is equivalent to $\CS\left(\frac{1}{8}\right)=\CS\left(\frac{3}{8}\right)$. Therefore, $\CS\left(\frac{x}{8}\right)=0$ for all odd $x$.
\end{proof}

Take $p,q\in \fS$. As always, we denote by $m=pq$ and $\chi$ the Dirichlet character associated with the extension $\BQ(\sqrt{m})/\BQ$. 

\begin{prop}\label{l-val-em}
We have
\begin{enumerate}
\item[(1)]$\frac{L(E^{(m)},1)}{\Omega_{E^{(m)}}}$ is an odd integer.
\item[(2)]$L(E^{(2m)},1)=0$.
\end{enumerate}
\end{prop}

\begin{proof}
Part (2) follows from the fact that $E^{(2m)}$ has root number $-1$. We now prove part (1). Since $\chi$ is even and $\CS\left(\frac{x}{m}\right)=\CS\left(\frac{m-x}{m}\right)$, applying (2) in Theorem \ref{thm3-3-1}, we obtain
\begin{equation}\label{3-4-f1}
\frac{L(E^{(m)},1)}{\Omega_{E^{(m)}}}=\sum^{\frac{m-1}{2}}_{x=1}\chi(k)\cdot 2\CS\left(\frac{x}{m}\right).
\end{equation}
Note that the Manin constant $\nu_E=1$ and $2\CS\left(\frac{x}{m}\right)\in\BZ$, we know that $L(E^{(m)},1)/\Omega_{E^{(m)}}$ is an integer. Since $\chi(x)=\pm1$ if \((x,m)=1\), we take both sides of \eqref{3-4-f1} modulo $2$, yielding
\[\frac{L(E^{(m)},1)}{\Omega_{E^{(m)}}}\equiv \sum_{1\leq x\leq \frac{m-1}{2},\, (x,m)=1}2\CS\left(\frac{x}{m}\right)\mod 2.\]
By applying the relation $\CS\left(\frac{x}{m}\right)=\CS\left(\frac{m-x}{m}\right)$ again, we obtain
\[\frac{L(E^{(m)},1)}{\Omega_{E^{(m)}}}\equiv \msp_{\mathds{1},m}(\CS)\mod 2.\]
Finally, applying (1) from Lemma \ref{lem3-3-2}, we have
\[\frac{L(E^{(m)},1)}{\Omega_{E^{(m)}}}\equiv \frac{1}{3}\left((a_p-2)(a_q-2)-(p-1)(q-1)\right)\mod 2.\]
Since $a_p$ and $a_q$ are odd, we conclude that $\frac{L(E^{(m)},1)}{\Omega_{E^{(m)}}}$ is an odd integer.
\end{proof}

\section{$2$-Adic Properties of Mazur-Tate Elements under Specializations}\label{sec4}

\subsection{$\mu$ and $\lambda$ Invariants}
Let $\BQ_\infty/\BQ$ be the cyclotomic $\BZ_2$-extension of $\BQ$, and let $\BQ_n$ denote the unique intermediate field in $\BQ_\infty$ such that $[\BQ_n:\BQ]=2^n$. Define $G_n={\rm Gal}(\BQ_n/\BQ)$, and let $\Lambda_n=\BZ_2[G_n]$ denote the group algebra associated with $G_n$.  
In this subsection, we assume that $n$ is a positive integer. We will review Pollack's reinterpretation of Kurihara's method for determining the $2$-adic valuations of the specializations of modular elements on finite-order characters of ${\rm Gal}(\BQ_\infty/\BQ)$. A key concept in Pollack’s approach is the use of Iwasawa invariants for the group ring $\Lambda_n=\BZ_2[G_n]$. While Pollack's work primarily focused on odd primes  (cf. \cite{pollack}), we will show that the same method can be extended to the case of the prime $2$. For the convenience of the reader, we will carefully review the results and provide a detailed proof specifically for the case of the prime $2$.

\medskip

 Let $I_n$ denote the augmentation ideal of $\Lambda_n$. We also define $\widetilde{\Lambda}_n=\BF_2[G_n]$ and let $\widetilde{I}_n$ be the augmentation ideal of $\widetilde{\Lambda}_n$.
 
 \begin{defn}
 For a nonzero element $g\in\Lambda_n$, we define the following: 
 \begin{enumerate}
 \item[(1)]The $\mu$-invariant of $g$ is the unique integer $\mu(g)$ such that 
 $g\in 2^{\mu(g)}\Lambda_n$ but $g\notin 2^{\mu(g)+1}\Lambda_n$.
 \item[(2)]The $\lambda$-invariant of $g$ is the unique integer $\lambda(g)$ such that the reduction of $2^{-\mu(g)}g$ modulo $2$ lies in $\widetilde{I}^{\lambda(g)}_n$ but does not lie in $\widetilde{I}^{\lambda(g)+1}_n$. 
 \end{enumerate}
 \end{defn}
 
 The $\mu$ and $\lambda$ invariants are not additive in general. However, they satisfy the following weaker properties, whose proofs we omit.
 
\begin{lem}\label{lem4-1-1}
For nonzero $f,g\in\Lambda_n$, we have
\begin{enumerate}
\item[(1)]$\mu(fg)\geq\mu(f)+\mu(g)$
\item[(2)]If $\mu(fg)=0$, then $\lambda(fg)=\lambda(f)+\lambda(g)$.
\end{enumerate}
\end{lem} 
  
\begin{lem}\label{lem4-1-2}
Let $\psi_n$ be a character of $G_n$ of order $2^n$. If $f\in\Lambda_n$ satisfies $\lambda(f)<2^{n-1}$, then
\[v(\psi_n(f))=\mu(f)+\frac{\lambda(f)}{2^{n-1}},\]
where $v$ denotes the normalized additive $2$-adic valuation on $\BQ_2(\zeta_{2^n})$ such that $v(2)=1$.
\end{lem}

\begin{proof}
Let $\gamma$ be a generator of $G_n$.
We can express $f\in\Lambda_n$ as 
\[f=2^{\mu(f)}\left((\gamma-1)^{\lambda(f)}u+2f_0\right),\]
where $u\in\Lambda^\times_n$ and $f_0\in\Lambda_n$. Since $\psi_n$ has order $2^n$, we know that $\psi_n(\gamma)$ is a primitive $2^n$-th root of unity. Thus
\[v(\psi_n(\gamma)-1)=\frac{1}{2^{n-1}}.\]
Observe that, when $\lambda(f)<2^{n-1}$, we have
\[\lambda(f)\cdot v(\psi_n(\gamma)-1)< v(2\psi_n(f_0)).\]
Therefore, the lemma follows.
\end{proof}

Define 
\[\fd_n=\sum_{x\in G_n, x^2=1}x\in\Lambda_n.\]
Note that the sum is taken over elements in the kernel of the projection $\pi_n:G_n\to G_{n-1}$. It is clear that $\mu(\fd_n)=0$.
As in the proof of Lemma \ref{lem4-1-2}, we let $\gamma$ be a generator of $G_n$. Then
\[\fd_n=1+\gamma^{2^{n-1}}\equiv (\gamma-1)^{2^{n-1}} \mod 2,\]
therefore, we obtain that $\lambda(\fd_n)=2^{n-1}$.

\begin{lem}\label{lem4-1-3}
For nonzero $f\in\Lambda_{n-1}$ and $g\in \Lambda_n$, we have
\begin{enumerate}
\item[(1)]$\pi_n\circ\nu_{n-1}(f)=2f$.
\item[(2)]$\nu_{n-1}\circ\pi_n(g)=\fd_n g$.
\end{enumerate}
\end{lem}

\begin{proof}
For part (1), let $f=\sum_{x\in G_{n-1}}c_x x$. We observe that $\nu_{n-1}(x)=x_0\fd_n$, where $x_0$ is any lift of $x$ in $G_n$, and $\pi_n(x_0y)=x$
for any $y\in {\rm ker}(\pi_n)$. Therefore, we have
\[\pi_n\circ\nu_{n-1}(f)=\sum_{x\in G_{n-1}}c_x\pi_n(x_0\fd_n)=2f.\]
Part (2) follows in a similar manner, and we omit the proof for brevity.
\end{proof}

\begin{prop}\label{prop4-1-4}
For nonzero elements $g\in\Lambda_n$ and $f\in \Lambda_{n-1}$, we have
\begin{enumerate}
\item[(1)]$\mu(\pi_n(g))\geq\mu(g)$. In particular, $\mu(\pi_n(g))=0$ implies $\mu(g)=0$.
\item[(2)]If $\mu(\pi_n(g))=\mu(g)$, then $\lambda(\pi_n(g))=\lambda(g)$.
\item[(3)]$\mu(\nu_{n-1}(f))=\mu(f)$.
\item[(4)]$\lambda(\nu_{n-1}(f))=2^{n-1}+\lambda(f)$.
\end{enumerate}
\end{prop}

\begin{proof}
Part (1) follows directly from the definition of $\mu$-invariants. Parts (2) and (3) follow by the same way as in \cite[Proposition 4.9]{pollack}. We now prove part (4). By considering $2^{-\mu(f)}f$ instead of $f$, we may assume that $\mu(f)=0$. From part (3), we know that $\mu(\nu_{n-1}(f))=0$. Let $f_0$ be any element in $\Lambda_n$ such that $\pi_n(f_0)=f$. From part (2) of Lemma \ref{lem4-1-3}, we obtain
\[\nu_{n-1}(f)=\fd_n f_0.\]
By part (2) of Lemma \ref{lem4-1-1}, we have $\lambda(\nu_{n-1}(f))=\lambda(\fd_n)+\lambda(f_0)$. Since $\mu(\pi_n(f_0))=\mu(f)=0$ by part (1), it follows that $\mu(f_0)=0$. From part (2), we obtain $\lambda(f)=\lambda(f_0)$. Thus, part (4) follows from the fact that $\lambda(\fd_n)=2^{n-1}$.
\end{proof}

\medskip

\subsection{$2$-Adic Properties of Modular Elements Associated with $E$ Under Specializations}
Let $\psi_n$ be a character of $G_n$ of order $2^n$. In the following, let $\gamma=\sigma_5$ be a generator of $G_n$, and assume that $\psi_n(\gamma)=\zeta_{2^n}$, where $\zeta_{2^n}$ is a primitive $2^n$-th root of unity. Since we are primarily concerned with the $2$-adic valuations of modular elements at finite order characters of $G_n$, and the modular elements have coefficients in $\BQ$, once we determine the $2$-adic valuation at $\psi_n$, we can deduce the $2$-adic valuations for the other specializations by applying the Galois action on the modular elements at $\psi_n$.

Let $E$ be the elliptic curve defined in the introduction.  From now until the end of this section, for any rational number $\beta=\frac{k}{l}$ with $(l,k)=1$ and $(l,3)=1$, we associate  $\CS(\beta)$ with $E$.  The following two lemmas are based on numerical computation.

\begin{lem}\label{lem4-2-0}
For $\CS\left(\frac{x}{16}\right)$ with $x\in[1,8]$ being an odd integer, we have
(1) $\CS\left(\frac{1}{16}\right)=\CS\left(\frac{5}{16}\right)=0$; (2) $\CS\left(\frac{3}{16}\right)=\CS\left(\frac{7}{16}\right)=-\frac{1}{2}$.
\end{lem}

\begin{lem}\label{lem4-2-1}
For $\CS\left(\frac{x}{32}\right)$ with $x\in[1,16]$ being an odd integer, we have
\begin{enumerate}
\item[(1)]$\CS\left(\frac{x}{32}\right)=\frac{-1}{2}$ for $x=3,7,9,13$.
\item[(2)]$\CS\left(\frac{x}{32}\right)=-1$ for $x=11,15$.
\item[(3)]$\CS\left(\frac{x}{32}\right)=0$ for $x=1,5$.
\end{enumerate}
\end{lem}

\medskip

Combining this with Lemma \ref{s(k/8)}, we can explicitly determine all the values $\CS\left(\frac{x}{32}\right)$ for integers $x\in[1,32]$.
From part (2) of Lemma \ref{lem3-2-2}, we know that $\xi_{\BQ_n}\in\BZ_2[G_n]$. Then
\[\psi_n(\xi_{\BQ_n})=\sum_{k\in(\BZ/2^{n+2}\BZ)^\times}\psi_n(k)\cdot\CS\left(\frac{k}{2^{n+2}}\right)=\msp_{\psi_n}(\CS)\in\BZ_2[\zeta_{2^n}]. \]
We define a sequence $\{q_n\}^\infty_{n=1}$ of integers as follows:
\begin{itemize}
\item $q_1=1$;
\item if $n$ is an odd integer and $n\geq3$, $q_n=\frac{2^n+7}{3}$;
\item if $n$ is a positive even integer, $q_n=\frac{2^n-1}{3}$.
\end{itemize}

\begin{prop}\label{prop4-2-2}
We have $\psi_1(\xi_{\BQ_1})=0$, and 
$\psi_n(\xi_{\BQ_n})=\frac{q_n}{2^{n-1}}$
for all integers $n\geq2$.
\end{prop}

\begin{proof}
From Lemma \ref{s(k/8)}, we obtain that $\xi_{\BQ_1}=0$. Therefore, $\psi_1(\xi_{\BQ_1})=0$. Next, we claim that $\mu(\xi_{\BQ_n})=0$ for all $n\neq1$. Recall that we fix $\gamma=\sigma_5$, a generator of $G_n$, and that $\psi_n(\gamma)=\zeta_{2^n}$, a primitive $2^n$-th root of unity. 
From Lemma \ref{lem4-2-1}, a direct computation shows that
\[\xi_{\BQ_3}=-\gamma^2-\gamma^3-2\gamma^4-2\gamma^5-\gamma^6-\gamma^7\]
and
\[\xi_{\BQ_2}=-\gamma^2-\gamma^3.\]
Thus, we have $\mu(\xi_{\BQ_2})=\mu(\xi_{\BQ_3})=0$. Now, we proceed by induction to establish the claim. For $n\geq2$, we assume that $\mu(\xi_{\BQ_n})=0$.
Recall the trace element $\tau_{\BQ_n}\in\BZ_2[G_n]$, and from part (1) of Lemmas \ref{lem3-2-2} and \ref{s(k/8)}, we have the relation
\begin{equation}\label{4-2-f0}
\pi_{n+2}(\xi_{\BQ_{n+2}})=-\nu_n(\xi_{\BQ_n})-2\tau_{\BQ_n}.
\end{equation}
Applying part (3) of Proposition \ref{prop4-1-4}, we obtain
\begin{equation}\label{4-2-f1}
\mu(\xi_{\BQ_n})=\mu(\nu_n(\xi_{\BQ_n}))=\mu(\nu_n(\xi_{\BQ_n})+2\tau_{\BQ_n})=\mu(\pi_{n+2}(\xi_{\BQ_{n+2}}))\end{equation}
for all $n\geq2$. By the induction hypothesis and part (1) of Proposition \ref{prop4-1-4}, the claim follows. 
We now apply Lemma \ref{lem4-1-2} to calculate the $2$-adic valuation of $\psi_n(\xi_{\BQ_n})$ for $n\neq3$. For $n\geq2$, observe that $\mu(\pi_{n+2}(\xi_{\BQ_{n+2}}))=\mu(\xi_{\BQ_{n+2}})=0$. From part (2) of Proposition \ref{prop4-1-4} and relation \eqref{4-2-f0}, we have $\lambda(\xi_{\BQ_{n+2}})=\lambda(\nu_n(\xi_{\BQ_n}))$. Applying part (4) of Proposition \ref{prop4-1-4}, we obtain the following relation
\begin{equation}\label{4-2-f2}
\lambda(\xi_{\BQ_{n+2}})=\lambda(\xi_{\BQ_n})+2^n
\end{equation}
for all $n\geq2$. Since
\[\xi_{\BQ_3}\equiv \gamma^2(\gamma-1)^5\mod 2\quad\textrm{ and }\quad \xi_{\BQ_2}\equiv \gamma^2(\gamma-1)\mod 2,\]
we can conclude that $\lambda(\xi_{\BQ_2})=1$ and $\lambda(\xi_{\BQ_3})=5$. Combining these with relation \eqref{4-2-f2}, a straightforward computation yields
\[\lambda(\xi_{\BQ_n})=q_n\qquad \forall n\geq2.\] 
Thus, the proposition follows from $\mu(\xi_n)=0$ and Lemma \ref{lem4-1-2}, except for the case $n=3$, where the condition $\lambda(\xi_{\BQ_3})<4$ fails. We now proceed to calculate $v(\psi_3(\xi_{\BQ_3}))$. Let $\zeta=\zeta_8$ for simplicity.  A direct computation gives
\[\psi_3(\xi_{\BQ_3})=-\zeta^2(1+\zeta)(1+\zeta^2)^2.\]
Thus, we find that $v(\psi_3(\xi_{\BQ_3}))=\frac{5}{4}$. This completes the proof of the proposition.
\end{proof}

\medskip

\subsection{$2$-Adic Properties of Modular Elements Associated with $(E,\chi)$ Under Specializations}
 Let $m=pq$, where $p$ and $q$ are distinct primes satisfying $p\equiv q\equiv 7\mod 12$, and such that $a_p$ and $a_q$ are odd. 
 
 We define a new modular element associated with $(E,\chi)$ for level $2^{n+2}m$ as follows
 \[\xi_{2^{n+2}}(m)=\sum_{x\in(\BZ/2^{n+2}m\BZ)^\times}\CS\left(\frac{x}{2^{n+2}m}\right)\tau_x,\]
where $\tau_x$ denotes the element of ${\rm Gal}(\BQ(\zeta_{2^{n+2}m})/\BQ)$ acting on $\zeta_{2^{n+2}m}$ by raising it to the $x$-th power. Recall that $\psi_n$ is the character of $G_n$ with order $2^n$. 
We can view  $\chi\psi_n$ as an even character of ${\rm Gal}(\BQ(\zeta_{2^{n+2}m})/\BQ)$. An easy calculation shows that $\chi\psi_n(\xi_{2^{n+2}}(m))$ lies in $\BZ[\zeta_{2^n}]$.

\begin{thm}\label{thm4-3-1}
Let $\rho_n=\chi\psi_n$. Let $q_n$ be the same sequence as in Proposition \ref{prop4-2-2}. Then the following statements hold:
\begin{enumerate}
\item[(1)]$\rho_1(\xi_{8}(m))=0$. 
\item[(2)]For each positive integer $n\geq2$,  $v(\rho_n(\xi_{2^{n+2}}(m)))=\frac{q_n}{2^{n-1}}$. 
\end{enumerate}
\end{thm}

In this subsection, we prove part (2) of Theorem \ref{thm4-3-1}  for $n\neq3$. The case of $n=3$ will be treated in Appendix A. Our approach for $n\neq3$ is distinct from that of Proposition \ref{prop4-2-2}. Specifically, we employ a congruence modulo $2$ to reduce Theorem \ref{thm4-3-1} to Proposition \ref{prop4-2-2}, a step in which the relation stated in part (2) of Lemma \ref{lem3-3-2} plays a role. 

\begin{proof}
Since the root number of $L(E^{(2m)},s)$ is $-1$, part (2) of Theorem \ref{thm3-3-1} gives $\rho_1(\xi_8(m))=0$.  Now consider the case where $n\neq1$ and $n\neq3$. 
Observe that $\rho_n$ is an even character. By part (1) of Proposition \ref{prop3-1} we have
\[\rho_n(\xi_{2^{n+2}}(m))=\sum^{2^{n+1}m}_{x=1}\psi_n(x)\chi(x)\cdot2\CS\left(\frac{x}{2^{n+2}m}\right)\]
Since $2\CS\left(\frac{x}{2^{n+2}m}\right)\in\BZ$ and $\chi$ takes values in $\{\pm1\}$ if \((x,m)=1\), reducing the above equality modulo $2$ yields
\[\rho_n(\xi_{2^{n+2}}(m))\equiv \sum_{1\leq x\leq 2^{n+1}m, (x,m)=1}\psi_n(x)\cdot2\CS\left(\frac{x}{2^{n+2}m}\right)\mod 2.\] 
Recall the definition of $\msp_{\psi_n,2^{n+2}m}(\CS)$. Because $\psi_n$ is even,  from part (1) of Proposition \ref{prop3-1}, we obtain
\begin{equation}\label{4-3-f1}
\rho_n(\xi_{2^{n+2}}(m))\equiv \msp_{\psi_n,2^{n+2}m}(\CS)\mod 2. 
\end{equation} 
Applying (2) in Lemma \ref{lem3-3-2}, together with the identity $\msp_{\psi_n}(\CS)=\psi_n(\xi_{\BQ_n})$, equation \eqref{4-3-f1} becomes
\begin{equation}\label{4-3-f2}
\rho_n(\xi_{2^{n+2}}(m))\equiv \left(\prod_{\ell\mid m}(a_\ell-\psi_n(\ell)-\overline{\psi}_n(\ell))\right)\cdot\psi_n(\xi_{\BQ_n})\mod 2.
\end{equation}
The assumption that $a_p$ and $a_q$ are odd implies 
$$v(a_\ell-\psi_n(\ell)-\overline{\psi}_n(\ell))=0$$
for $\ell=p$ or $q$.  Now the case $n\neq3$ of the theorem follows from \eqref{4-3-f2}, Proposition \ref{prop4-2-2} and the fact that $q_n<2^{n-1}$ for all $n\neq3$.
\end{proof}

\section{Rank-One Twists}\label{sec5}

In this section, we determine the algebraic rank and the analytic rank of $E^{(m)}$ over $\BQ$ and $\BQ_1$, respectively. Combining these results with the non-vanishing conclusions from the previous section, we then proceed to prove the main theorem stated in the introduction.

\subsection{Algebraic and Analytic Rank of $E^{(m)}$ over $\BQ$}

We determine the rank of $E^{(m)}(\BQ)$ using the $2$-descent method. Let $\fK$ be a number field. For each place $v$ of $\fK$, denote by $\fK_v$ the completion of $\fK$ at $v$. Let $A$ be an elliptic curve defined over $\fK$, and let
\[H^1_f(\fK_v,A_2)=\Im\left(A(\fK_v)/2A(\fK_v)\to H^1(\fK_v,A_2)\right)\]
denote the image of $A(\fK_v)/2A(\fK_v)$ under the local Kummer map. The $2$-Selmer group of $A/\fK$ is then defined as 
\[{\rm Sel}_2(A/\fK)={\rm ker}\left(H^1(\fK,A_2)\to\prod_{v}\frac{H^1(\fK_v,A_2)}{H^1_f(\fK_v,A_2)}\right),\]
where the product runs over all places of $\fK$.

\medskip

\begin{lem}\label{lem5-1}
Let $\fL/\fK$ be a quadratic extension of number fields, and denote by $A^\fL$ the quadratic twist of $A$ by $\fL/\fK$. Suppose at least one of the following conditions holds for a place $v$ of $\fK$:
\begin{enumerate}
\item[(i)]$v$ splits in $\fL/\fK$;
\item[(ii)]$v\nmid 2\infty$ and $A(\fK_v)_2=0$; 
\item[(iii)]$A$ has multiplicative reduction at $v$, $\fL/\fK$ is unramified at $v$, and ${\rm ord}_v(\Delta_A)$ is odd, where ${\rm ord}_v$ is the normalized additive valuation on $\fK_v$ with $\ord_v(\varpi_v)=1$ for any uniformizer $\varpi_v$ of $\fK_v$;
\item[(iv)]$v=\BR$ and $(\Delta_A)_v<0$;
\item[(v)]$v$ is a prime where $A$ has good reduction and $v$ is unramified in $\fL/\fK$.
\end{enumerate}
Then 
$$H^1_f(\fK_v,A^\fL_2)=H^1_f(\fK_v,A_2).$$ 
In particular, if this equality holds for every place
$v$ of $\fK$, we have
\[{\rm Sel}_2(A^\fL/\fK)\simeq {\rm Sel}_2(A/\fK).\]
\end{lem}

\begin{proof}
The first claim follows by \cite[Lemma 2.10]{MR-inv}. Because the $G_\fK$-modules $A^\fL_2$ and $A_2$ are canonically isomorphic (see \cite[Subsection 1.4]{MR-inv}), we can regard ${\rm Sel}_2(A^\fL/\fK)$ as a subgroup of $H^1(\fK,A_2)$. Consequently, the identical local conditions imply
 \[{\rm Sel}_2(A^\fL/\fK)\simeq {\rm Sel}_2(A/\fK).\]
\end{proof}

\begin{prop}\label{prop5-2}
We have 
\[{\rm ord}_{s=1}L(E^{(m)},s)=0={\rm rank}(E^{(m)}(\BQ)),\]
and $\Sha(E^{(m)}/\BQ)(2)=0$. In particular, the rank part of the Birch-Swinnerton-Dyer conjecture holds for $E^{(m)}/\BQ$.
\end{prop}

\begin{proof}
By Proposition \ref{l-val-em}, $L(E^{(m)},1)\neq0$. A direct application of the Gross-Zagier theorem and Kolyvagin's theorem then implies that both $E^{(m)}(\BQ)$ and $\Sha(E^{(m)}/\BQ)$ are finite. To show that $\Sha(E^{(m)}/\BQ)(2)=0$, we apply Lemma \ref{lem5-1} with $A=E^{(m)}$, $\fK=\BQ$, $\fL=\BQ(\sqrt{m})$. It follows that 
$$H^1_f(\BQ_v,E^{(m)}_2)=H^1_f(\BQ_v,E_2)$$
for every places $v$ of $\BQ$. Since ${\rm Sel}_2(E/\BQ)=0$, we obtain ${\rm Sel}_2(E^{(m)}/\BQ)=0$. This completes the proof. 
\end{proof}

\medskip

\subsection{Algebraic and Analytic Rank of $E^{(m)}$ over $\BQ_1$}
We determine the rank of $E^{(m)}$ over $\BQ_1$ by showing that certain Heegner points are non-torsion. For a number field $\fK\subset \BQ_\infty$,  denote by $L(E^{(m)}/\fK,s)$ the $L$-series associated with $E^{(m)}$ over $\fK$. It is clear that $L(E/\fK,s)$ is holomorphic on the whole complex plane. 
Since $\BQ_1=\BQ(\sqrt{2})$, we have
\[L(E^{(m)}/\BQ_1,s)=L(E^{(2m)},s)L(E^{(m)},s).\]
As shown in Proposition \ref{prop5-2}, we have $L(E^{(m)},1)\neq0$. In this subsection, we will use the theory of Heegner points to show that, under certain conditions, the $L$-function $L(E^{(2m)},s)$ has a simple zero at $s=1$ and that $E^{(2m)}(\BQ)$ has rank $1$. Considering the Galois action of ${\rm Gal}(\BQ_1/\BQ)$ on $E^{(m)}(\BQ_1)$, we conclude that the Mordell-Weil rank of $E^{(m)}(\BQ_1)$ is $1$. 

\begin{prop}\label{prop5-3}
We have
\[{\rm Sel}(E^{(m)}/\BQ_1)^\vee=\mathbb{Z}_2\]
\end{prop}

\begin{proof}
Apply Lemma \ref{lem5-1} with $A=E^{(m)}$, $\fK=\BQ_1$, $\fL=\BQ_1(\sqrt{m})$. It follows that 
$$H^1_f(\BQ_{1,v},E^{(m)}_2)=H^1_f(\BQ_{1,v},E_2)$$
for every places $v$ of $\BQ_1$. Since ${\rm Sel}_2(E/\BQ_1)=\mathbb{Z}/2\mathbb{Z}$, we obtain ${\rm Sel}_2(E^{(m)}/\BQ_1)=\mathbb{Z}/2\mathbb{Z}$.

Since \(E^{(2m)}\) has root number \(-1\), by the parity in \cite[Theorem 1.4]{DD} the corank of \({\rm Sel}(E^{(m)}/\BQ_1)\) must be 1. 
\end{proof}

\medskip

The following criterion for establishing the non‑torsionness of Heegner points is due to Kriz and Li (see \cite[Theorem 1.20]{KL-sigma}). We restate it here as a theorem.
 
\begin{thm}\label{thm5-2-0}
Let $A/\BQ$ be an elliptic curve of conductor $N$. Suppose $\ell$ is an odd prime such that $A_\ell$ is a reducible ${\rm Gal}(\overline{\BQ}/\BQ)$-module, and write 
$$A^{\rm ss}_\ell\simeq\BF_\ell(\psi)\oplus\BF_\ell(\psi^{-1}\omega)$$ 
for some character $\psi:{\rm Gal}(\overline{\BQ}/\BQ)\to {\rm Aut}(\BF_\ell)\simeq\mu_{\ell-1}$ and the mod $\ell$ cyclotomic character $\omega$. Assume that
\begin{enumerate}
\item[(1)]$\psi(\ell)\neq1$ and $(\psi^{-1}\omega)(\ell)\neq1$;
\item[(2)]$A$ has no prime of split multiplicative reduction;
\item[(3)]If $\ell_0\neq\ell$ is a prime at which $A$ has additive reduction,, then $\psi(\ell_0)\neq1$ and $(\psi^{-1}\omega)(\ell_0)\neq1$.
\end{enumerate}
Let $K$ be an imaginary quadratic field satisfying the Heegner hypothesis for $N$, and denote by $P\in A(K)$ the associated Heegner point. Assume further that $\ell$ splits in $K$ and that the generalized Bernoulli numbers satisfy
\[B_{1,\psi_0^{-1}\chi_K}\cdot B_{1,\psi_0\omega^{-1}}\not\equiv0\mod \ell,\]
where $\chi_K$ is the quadratic character associated to $K/\BQ$, and
$\psi_0$ is $\psi$ or $\psi\chi_K$ according as $\psi$ is even or odd. Then $P$ is a non-torsion point in $A(K)$, and $A/K$ has analytic and algebraic rank $1$.
\end{thm}

\medskip

Recall that the newform associated with $E$ is given by $\sum_{n\geq1}a_nq^n$, where $q=e^{2\pi iz}$. For a squarefree integer $l$, we denote by $h(l)$ the class number of the quadratic number field $\BQ(\sqrt{l})$. 
 
\begin{prop}\label{prop5-2-1}
Let $p$ and $q$ be distinct primes satisfying $p\equiv q\equiv 7\mod 12$ and such that $a_p$ and $a_q$ are odd. Set $m=pq$ and assume further that
\begin{enumerate}
\item[(1)]$\left(\frac{-2p}{q}\right)=1$
\item[(2)]$3$ is coprime to both $h(-q)$ and $h(-6pq)$.
\end{enumerate}
Then 
\[\ord_{s=1}L(E^{(2m)},s)=1={\rm rank}(E^{(2m)}(\BQ)).\]
In particular, we have
\[{\rm ord}_{s=1}L(E^{(m)}/\mathbb{Q}_1,s)=1={\rm rank}(E^{(m)}(\BQ_1)).\]
\end{prop}

\begin{proof}
The "in particular" statement follows from the first claim and the remark made at the beginning of this subsection. To prove the first claim we apply Theorem \ref{thm5-2-0} with
$A=E^{(3q)}$, $\ell=3$ and $K=\BQ(\sqrt{-2p})$. Indeed, by \cite[Proof of Lemma 10.3]{KL-sigma}, the module $E^{(3q)}_3$ is reducible as a ${\rm Gal}(\overline{\BQ}/\BQ)$-module, and the character $\psi$ is the quadratic character associated with $\BQ(\sqrt{3q})/\BQ$.  The following observations are made:
\begin{enumerate}
\item[(a)]Since $3q\equiv 3\mod 9$, the condition (1) of Theorem \ref{thm5-2-0} holds (see \cite[Proof of Theorem 10.6]{KL-sigma}). 
\item[(b)]The curve $E^{(3q)}$ has bad reduction only at the primes $3$ and $q$, which are additive. 
\item[(c)]For the prime $q$, we have $\psi(q)=(\psi^{-1}\omega)(q)=0$,  which implies the condition (3) of Theorem \ref{thm5-2-0} is satisfied.
\item[(d)]The Heegner hypothesis is equivalent to the condition $\left(\frac{-2p}{q}\right)=1$.
\item[(e)]From \cite[Proof of Theorem 10.6]{KL-sigma} and the class number formula, assumption (2) of the proposition is equivalent to 
the generalized Bernoulli number assumption in Theorem \ref{thm5-2-0}.
\end{enumerate}
Thus $E^{(3q)}/K$ has analytic and algebraic rank $1$. Using the identity
\[L(E^{(3q)}/K,s)=L(E^{(-6pq)},s)L(E^{(3q)},s),\]
together with a root number consideration, we conclude that $E^{(-6m)}/\BQ$ also has analytic rank $1$. Since $E^{(-6m)}$ and $E^{(2m)}$ are isogenous, $E^{(2m)}$ has analytic rank $1$. The first claim now follows from the Gross-Zagier theorem and Kolyvagin's theorem.  
\end{proof}

\subsection{Proof of the Main Theorem}

We now prove Theorem \ref{thm2}. Part (i) follows from Proposition \ref{l-val-em} and Coates-Wiles theorem. We prove part (ii).
By Theorem \ref{thm4-3-1}, we have
$L(E^{(m)},\rho,1)\neq0$ for every finite-order character $\rho$ of ${\rm Gal}(\BQ_\infty/\BQ)$ except $\psi_1$. Proposition \ref{prop5-2-1} gives
${\rm ord}_{s=1}L(E^{(m)},\psi_1,s)=1$. For each integer $n\geq1$, we use the identity
\[L(E^{(m)}/\mathbb{Q}_n,s)=\prod_\eta L(E^{(m)},\eta,s),\]
where $\eta$ runs over all characters of ${\rm Gal}(\BQ_n/\BQ)$. This yields 
$${\rm ord}_{s=1}L(E^{(m)}/\BQ_n,s)=1.$$

On the other hand, $E^{(m)}$ has complex multiplication by $F=\BQ(\sqrt{-3})$. Let $\CO_F$ be its ring of integers. Write $F_n=F.\BQ_n$ for the compositum of $F$ and $\BQ_n$. The algebraic rank of $E^{(m)}(\BQ_n)$ is equal to the $\CO_F$-rank of $E^{(m)}(F_n)$. We claim that $E^{(m)}(F_n)$ has $\CO_F$-rank $1$. 
Indeed, consider the $\BC[{\rm Gal}(F_n/F)]$-module $E^{(m)}(F_n)\otimes\BC$. Note that
\[E^{(m)}(F_n)\otimes\BC=\bigoplus_{\eta}(E^{(m)}(F_n)\otimes\BC)^\eta,\] 
where $(E^{(m)}(F_n)\otimes\BC)^\eta$ denotes the $\eta$-isotypical component of $E^{(m)}(F_n)\otimes\BC$, and $\eta$ runs over all characters of ${\rm Gal}(F_n/F)$. Applying the main results of Appendix B together with Theorem \ref{thm4-3-1}, we obtain
\[(E^{(m)}(F_n)\otimes\BC)^\eta=0\]
for all $\eta\neq\psi_1$. By Proposition \ref{prop5-2-1} , we have
$(E^{(m)}(F_n)\otimes\BC)^{\psi_1}\simeq E^{(2m)}(F)\otimes\BC$, which is isomorphic to $\CO_F\otimes\BC$.
Hence, ${\rm rank}(E^{(m)}(\BQ_n))=1$ for all $n\geq1$. In particular, ${\rm rank}(E^{(m)}(\BQ_\infty))=1$. 

Finally, since $E^{(m)}$ has good supersingular reduction at $2$, \cite[Theorem 2.14]{CS} implies
\[{\rm rank}_\Lambda \Sel(E^{(m)}/\BQ_\infty)^\vee=1.\]
Because a finitely generated $\BZ_2$-module must be $\Lambda$-torsion, it follows that $\Sha(E/\BQ_\infty)(2)$ has infinite $2$-rank and is of infinite order. This completes the proof of Theorem \ref{thm2}.

\section{Appendix A: Supplement to the Proof of Theorem \ref{thm4-3-1}}\label{sec7}

This appendix treats the case $n=3$ of Theorem \ref{thm4-3-1}. 
For a rational number $\beta=\frac{k}{l}$ with $(k,l)=1$ and $(l,3)=1$, let $\CS(\beta)$ be the real part of the modular symbol associated with $E$.

\begin{prop}\label{ap-prop1}
Let $p$ and $q$ be two distinct primes such that $p\equiv q\equiv 7\mod 12$ and satisfying $a_p$ and $a_q$ are odd. Set $m=pq$. Define
\[\xi_{32}(m)=\sum_{x\in(\BZ/32m\BZ)^\times}\CS\left(\frac{x}{32m}\right)\tau_x,\]
where $\tau_x\in{\rm Gal}(\BQ(\zeta_{32m})/\BQ)$ satisfying $\tau_x(\zeta_{32m})=\zeta^x_{32 m}$. Let $\chi$ be the Dirichlet character attached to $\BQ(\sqrt{m})/\BQ$, and let $\psi_3$ be the faithful character of $G_3={\rm Gal}(\BQ_3/\BQ)$ with $\psi_3(\gamma)=\zeta_{8}$ where $\gamma=\sigma_5$ is a fixed generator of $G_3$. Let $\rho_3=\chi\psi_3$. Then
\[v(\rho_3(\xi_{32}(m)))=\frac{5}{4}.\]  
\end{prop}

By the proof of Theorem \ref{thm4-3-1},  we know that $\rho(\xi_{32}(m))$ lies in $\BZ_2[\zeta_8]$. Proposition \ref{prop4-2-2} together with congruence \eqref{4-3-f2} yield the following estimate.

\begin{lem}\label{ap-lem1}
We have $v(\rho_3(\xi_{32}(m)))\geq1$.
\end{lem}

Let $\zeta=\zeta_8$ for simplicity. Since $\BZ_2[\zeta]$ has rank $4$ over $\BZ$,  we assume that
\begin{equation}\label{ap-f1}
\rho_3(\xi_{32}(m))=c_1\zeta+c_2\zeta^2+c_3\zeta^3+c_4\zeta^4,
\end{equation}
where $c_i$ (for $1\leq i\leq 4)$ are integers.
The key part of the proof for Proposition \ref{ap-prop1} lies in establishing the following higher congruences between the coefficients $c_i$.

\begin{lem}\label{ap-lem2}
The following higher congruence relations hold:
\[c_1+c_4\equiv 2\sum_{x=\pm5,\pm15}\CS\left(\frac{x}{32}\right)\equiv0\mod4\]
\[c_2+c_3\equiv 2\sum_{x=\pm3,\pm7}\CS\left(\frac{x}{32}\right)\equiv 0\mod 4\]
\[c_1+c_2\equiv 2\sum_{x=\pm5,\pm7}\CS\left(\frac{x}{32}\right)\equiv 2\mod 4\]
\[c_3+c_4\equiv 2\sum_{x=\pm3,\pm15}\CS\left(\frac{x}{32}\right)\equiv 2\mod 4.\]
\end{lem}

The challenging part of proving each of these congruences lies in establishing the first congruence in each case. Once this is done, the subsequent congruences follow easily from Lemma \ref{lem4-2-1}. The strategy to prove the first congruence involves an application of the Hecke action on modular symbols, alongside the relation in \eqref{3-3-f}. We will provide the details of this after completing the proof of Proposition \ref{ap-prop1}.

\begin{proof}[Proof of Proposition \ref{ap-prop1}]
Solving the four congruences in Lemma \ref{ap-lem2}, we obtain the following relations for the coefficients $c_i$ modulo $4$:
\[c_4\equiv -c_1\mod 4,\quad c_2\equiv 2-c_1\mod 4,\quad c_3\equiv c_1-2\mod 4.\]
Substituting these expressions into equation \eqref{ap-f1}, we get the following congruence 
\[\rho_3(\xi_{32}(m))\equiv c_1\zeta+(2-c_1)\zeta^2+(c_1-2)\zeta^3-c_1\zeta^4\mod 4.\]
Additionally, we have the following cases depending on the value of $c_1$ modulo $4$
\[\rho_3(\xi_{32}(m))\equiv\begin{cases}2\zeta^2(1+\zeta) \mod 4\quad &\textrm{if $c_1\equiv 0\mod 4$,} \\ \zeta(1+\zeta)(1-\zeta^2) \mod 4\quad &\textrm{if $c_1\equiv 1\mod 4$,}\\  2\zeta(1+\zeta^3)\mod 4 \quad &\textrm{if $c_1\equiv 2\mod 4$,}\\ -\zeta(1+\zeta)(1-\zeta^2)\mod 4 \quad &\textrm{if $c_1\equiv 3\mod 4$.}\end{cases}\]
Therefore, the valuation $v(\rho_3(\xi_{32}(m)))$ is either $\frac{3}{4}$ or $\frac{5}{4}$, depending on whether $c_1$ is odd or even. However, from Lemma \ref{ap-lem1}, we must have $v(\rho_3(\xi_{32}(m)))=\frac{5}{4}$. The proposition follows.
\end{proof}

In the final part of this appendix, we prove the following congruence
\begin{equation}\label{ap-f2}
c_1+c_4\equiv 2\sum_{x=\pm5,\pm15}\CS\left(\frac{x}{32}\right)\mod4, 
\end{equation}
The other congruences in Lemma \ref{ap-lem2} can be shown in a similar way.
For any odd integer $k$, let $k'$ denote an integer which is the inverse of $n$ modulo $32$. 

\begin{lem}\label{ap-lem3}
For any odd integer $k$, we have
\[\sum_{x=\pm5 k,\pm15k;\, 1\leq x\leq 32}\CS\left(\frac{x}{32}\right)+\sum_{x=\pm5 k',\pm15k';\, 1\leq x\leq 32}\CS\left(\frac{x}{32}\right)=-4.\]
\end{lem}

\begin{proof}
The lemma follows from Lemma \ref{lem4-2-1} and the following table
\begin{center}
\begin{scriptsize}
\begin{table}[H]
\renewcommand{\arraystretch}{2}
\center
\begin{tabular}{|c|c|c|c|c|c|c|c|c|}
    \hline \(k\) & \(1\) & \(3\) & \(5\) & \(7\) & \(9\) & \(11\) & \(13\) & \(15\) \\
\hline \(\sum\limits_{x=\pm 5k,\pm 15k;\, 1\leq x\leq 32} \CS\left(\frac{x}{32}\right)\) & -2 & -3 & -3 & -2 & -2 & -1 & -1 & -2 \\
    \hline
\end{tabular}.
\end{table}
\end{scriptsize}
\end{center}
\end{proof}

\medskip

\begin{proof}[Proof of \eqref{ap-f2}]
By definition, we have
\[\rho_3(\xi_{32}(m))=\sum_{1\leq k\leq 32m,\, (k,2m)=1}\chi(k)\psi_3(k)\CS\left(\frac{k}{32m}\right).\]
Here, $\psi_3$ is regarded as a character on $(\BZ/32m\BZ)^\times$ by the composite of maps
\[\left(\BZ/32m\BZ\right)^\times\to \left(\BZ/32\BZ\right)^\times/\{\pm1\}\to \overline{\BQ}^\times,\]
where the first map is the quotient map and the second map is given by $\psi_3$, since $G_n$ is isomorphic to $(\BZ/32\BZ)^\times/\{\pm1\}$.
Thus, the equation \eqref{ap-f1} holds under the identification $\psi_3(5\mod 32)=\zeta$. 

By comparing both sides of \eqref{ap-f1}, we obtain
\[c_1\zeta+c_4\zeta^4=\sum_{\substack{1\leq k\leq 32m \\ k\equiv \pm1,\pm5,\pm11,\pm15\mod 32}}\chi(k)\psi_3(k)\CS\left(\frac{k}{32m}\right).\]
Since $m$ is odd, we can write $k=bm+32a$ with $a\in(\BZ/m\BZ)^\times$ and $b\in(\BZ/32\BZ)^\times$. Then
\begin{equation}\label{ap-f3}
c_1\zeta+c_4\zeta^4=\sum^m_{a=1}\sum_{b\in\fB}\chi(32a)\psi_3(bm)\CS\left(\frac{a}{m}+\frac{b}{32}\right),
\end{equation}
where $\fB$ denotes the set $\{\pm m',\pm5m',\pm 11m',\pm 15 m'\mod 32\}$. For $i=1,5,11,15$, we define $\fB_i=\{\pm im'\mod 32\}$, and consequently 
\[\fB=\bigsqcup_{i=1,5,11,15}\fB_i.\] 
By comparing the coefficients of $\zeta$ and $\zeta^4$ on both sides of equation \eqref{ap-f3}, we obtain the following expressions for $c_1$ and $c_4$
\[c_1=\left(\sum^m_{a=1}\sum_{b\in\fB_5}-\sum^m_{a=1}\sum_{b\in\fB_{11}}\right)\chi(32a)\CS\left(\frac{a}{m}+\frac{b}{32}\right)\]
and
\[c_4=\left(\sum^m_{a=1}\sum_{b\in\fB_{15}}-\sum^m_{a=1}\sum_{b\in\fB_{1}}\right)\chi(32a)\CS\left(\frac{a}{m}+\frac{b}{32}\right).\]
Next, using the relation \eqref{3-3-f} with $\beta=\frac{2a}{m}+\frac{b}{16}$ for $b\in\fB_{11}\cup\fB_{1}$ and $\ell=2$, we obtain
\[c_1+c_4=\chi(2)\sum^m_{a=1}\chi(a)\sum_{b\in\fB_5\cup\fB_{15}}\left(2\CS\left(\frac{a}{m}+\frac{b}{32}\right)+\CS\left(\frac{4a}{m}+\frac{b}{8}\right)-\CS\left(\frac{1}{2}\right)\right).\]
Since $\chi$ is nontrivial, we know that 
\[\sum^m_{a=1}\chi(a)\CS\left(\frac{1}{2}\right)=0.\] 
Additionally, we note that $\CS\left(\frac{4a}{m}+\frac{b}{8}\right)$ depends only on $b$ modulo $8$. Applying \eqref{3-3-f} again for $\ell=2$, we obtain
\[\sum_{b\in\fB_{5}\cup\fB_{15}}\CS\left(\frac{4a}{m}+\frac{b}{8}\right)=2\CS\left(\frac{64a}{m}\right)+2\CS\left(\frac{16a}{m}\right).\]
Since $m$ is odd and $\chi$ is a quadratic character, for $x=16$ or $64$, we have
\[\sum^m_{a=1}\chi(a)\CS\left(\frac{xa}{m}\right)=\sum^m_{a=1}\chi(a)\CS\left(\frac{a}{m}\right).\]
Combining these facts, we obtain the equation
\begin{equation}\label{ap-f4}
c_1+c_4=\chi(2)\sum^m_{a=1}\chi(a)\left(4\CS\left(\frac{a}{m}\right)+\sum_{b\in\fB_{5}\cup\fB_{15}}2\CS\left(\frac{a}{m}+\frac{b}{32}\right)\right).
\end{equation}
Since $\chi$ is even, and by part (1) of Proposition \ref{prop3-1}, we have
\[\sum^m_{a=1}\chi(a)\CS\left(\frac{a}{m}\right)=\sum^{(m-1)/2}_{a=1}\chi(a)\cdot 2\CS\left(\frac{a}{m}\right),\]
which is an integer. Taking both sides of \eqref{ap-f4} modulo $4$, we have
\begin{equation}\label{ap-f5}
c_1+c_4\equiv \chi(2)\sum^m_{a=1}\chi(a)\sum_{b\in \fB_5\cup\fB_{15}}2\CS\left(\frac{a}{m}+\frac{b}{32}\right)\mod 4.
\end{equation}
Note that the sum in the  right hand side equals
\[\sum\limits_{a=1}^{(m-1)/2} \chi(a) \sum\limits_{b\in \mathfrak{B}_5\cup \mathfrak{B}_{15}} 2\bigg( \mathcal{S}(\frac{a}{m}+\frac{b}{32})+\mathcal{S}(\frac{-a}{m}+\frac{b}{32}) \bigg)\]
which is further equal to
\begin{tiny}
\[\sum\limits_{a=1}^{(m-1)/2} \chi(a) \sum\limits_{\substack{b\equiv5m',15m'\mod32\\1\leq b\leq 32}} 2\bigg( \mathcal{S}(\frac{a}{m}+\frac{b}{32})+\mathcal{S}(\frac{-a}{m}+\frac{b}{32}) +\mathcal{S}(\frac{a}{m}+\frac{-b}{32})+\mathcal{S}(\frac{-a}{m}+\frac{-b}{32})\bigg).\]
\end{tiny}
Observe that the summands in the summation of \(b\) are even integers. From the fact that 
$\chi(2)$ and $\chi(a)$ takes values in $\{\pm1\}$ if \((a,m)=1\), the congruence \eqref{ap-f5} becomes
\[c_1+c_4\equiv \sum_{1\le a\le m,(a,m)=1} \sum_{b\in \fB_5\cup\fB_{15}}2\CS\left(\frac{a}{m}+\frac{b}{32}\right)\mod 4.\]
Note that $m=pq$, where $p$ and $q$ are distinct primes. We can write $a=pv+qu$ with $v\in(\BZ/q\BZ)^\times$ and $u\in(\BZ/p\BZ)^\times$. The above congruence becomes
\[c_1+c_4\equiv2\sum_{b\in\fB_{5}\cup\fB_{15}}\sum^{p-1}_{u=1}\sum^{q-1}_{v=1}\CS\left(\frac{u}{p}+\frac{v}{q}+\frac{b}{32}\right)\mod 4.\]
By applying relation \eqref{3-3-f} iteratively for $\ell=q$ and $\ell=p$, using the facts that $a_p$ and $a_q$ are odd, and that
\[2\sum_{b\in\fB_5\cup\fB_{15}}\sum_{k\in\BZ/q\BZ}\CS\left(\frac{k}{q}\right)\equiv 0\mod 4,\]
we obtain the following congruence
\[\begin{aligned}c_1+c_4\equiv2&\sum_{b\in \fB_{5}\cup\fB_{15}}\left(\CS\left(\frac{mb}{32}\right)+\CS\left(\frac{pmb}{32}\right)+\CS\left(\frac{qmb}{32}\right)
+\CS\left(\frac{pb}{32}\right)+\CS\left(\frac{qb}{32}\right)\right)\\+&2\sum_{b\in\fB_{5}\cup\fB_{15}}\sum_{l\mid m}\CS\left(\frac{l^2b}{32}\right) \mod4\end{aligned}\]
Let $\msq$ denote the right hand side of the above congruence. A direct computation shows that
\[\begin{aligned}\msq\equiv&2\sum_{b=\pm 5,\pm 15}\CS\left(\frac{b}{32}\right)\\
+&2\sum_{k\mid m}\left(\sum_{\substack{b=\pm 5k,\pm 15k\\ 1\leq b\leq 32}}\CS\left(\frac{b}{32}\right)+\sum_{\substack{b=\pm 5k',\pm 15k'\\1\leq b\leq32}}\CS\left(\frac{b}{32}\right)\right)\\
+&2\left(\sum_{\substack{b=\pm 5pq',\pm 15pq'\\1\leq b\leq 32}}\CS\left(\frac{b}{32}\right)+\sum_{\substack{b=\pm 5p'q,\pm 15p'q\\1\leq b\leq 32}}\CS\left(\frac{b}{32}\right)\right) \mod 4.
\end{aligned}\] 
By Lemma \ref{ap-lem3}, we complete the proof for the congruence \eqref{ap-f2}.
\end{proof}

\section{Appendix B: An Equivariant Coates-Wiles Theorem}\label{sec8}

Let $K$ be an imaginary quadratic field, and let $\CO_K$ denote its ring of integers.
Let $A$ be an elliptic curve defined over $K$ with complex multiplication by $\CO_K$. Let $\CK$ be a finite abelian extension of $K$, and let $\eta$ be a character of ${\rm Gal}(\CK/K)$. Denote by $\psi_A$ the Hecke character of $K$ associated to $A$, and by $L(\overline{\psi_A\eta},s)$ the $L$-series associated to the complex conjugate of $\psi_A\eta$. By Deuring's theorem, $L(\overline{\psi_A\eta},s)$ is holomorphic on the whole complex plane.
For a $\BC[{\rm Gal}(\CK/K)]$-module $\fM$ and a character $\chi$ of ${\rm Gal}(\CK/K)$, we denote by $\fM^\chi$ the $\chi$-part of $\fM$.

\begin{thm}\label{ap2-thm1}
Assume that $L(\overline{\psi_A\eta},1)\neq0$, then $(A(\CK)\otimes\BC)^\eta=0$.
\end{thm}    

The idea for the above theorem follows from the Iwasawa theoretic method for $\psi_A\eta$ introduced by Coates, Wiles and Rubin. Specifically, we choose an odd prime $\ell$ satisfying:
\begin{enumerate}
\item[(1)]$\ell$ splits in $K$, say $\ell\CO_K=\fL^*\cdot\fL$;
\item[(2)] $A$ has good ordinary reduction at both $\fL$ and $\fL^*$;
\item[(3)]$\ell$ is coprime to $[\CK:K]$ and unramified in $\CK/\BQ$.
\end{enumerate}
The proof then applies Iwasawa theory to the tower $\CK(E(\fL))/K$. Here, for an $\CO_K$-module $\fM$ and an element $\alpha\in \CO_K$, we denote by $\fM_\alpha$ the $\alpha$-torsion subgroup of $\fM$. For an ideal $\fb\subset\CO_K$,  we define $\fM_\fb=\cap_{\alpha\in\fb}\fM_\alpha$ and set $\fM(\fb)=\cup_{n\geq1}\fM_{\fb^n}$.

Let $\fM$ be a $\BZ_\ell[\fG]$-module, where $\fG$ is a finite abelian group whose order is prime to
$\ell$. For each character $\chi$ of $\fG$, we define $\fM^\chi$ to be the $\chi$-part of $\fM$. For a precise definition, see \cite{rubin}. Note that $\fM^\chi$ is a direct summand of $\fM$. 

For a finite extension $\fK$ of $K$, we define the \(\fL\)-power Selmer group $\Sel_\fL(A/\fK)$ by the exact sequence
\[0\to \Sel_\fL(A/\fK) \to H^1(\fK,A(\fL))\to \prod_{v}H^1(\fK_v,A)(\fL),\]
where the product runs over all places of $\fK$. 
Observe that $\Sel_\fL(A/\CK)$ is a $\BZ_\ell[{\rm Gal}(\CK/K)]$-module.
Theorem \ref{ap2-thm1} is a consequence of the following proposition.

\begin{prop}\label{ap2-prop2}
If $L(\overline{\psi_A\eta},1)\neq0$, then $\Sel_\fL(A/\CK)^\eta$ is finite.
\end{prop}

For a finite extension $\fK/K$, we define another Selmer group $\Sel'_\fL(A/\fK)$ by the exact sequence
\[0\to \Sel'_\fL(A/\fK)\to H^1(\fK,A(\fL))\to \prod_{v\nmid \fL}H^1(\fK_v,A)(\fL), \]
where the product runs over all primes of $\fK$ not lying above $\fL$. Let $\CF=\CK(A_\fL)$ and $\Delta={\rm Gal}(\CF/\CK)$. For a group $\fG$ and a $\fG$-module $\fM$, we denote by $\fM^\fG$ the submodule of $\fG$-invariant elements and $\fM_\fG$ for the maximal quotient of $\fM$ on which $\fG$ acts trivially.

\begin{lem}\label{ap2-lem3}
The following statements hold:
\begin{enumerate}
\item[(1)]The index of $\Sel_\fL(A/\CK)$ in $\Sel'_\fL(A/\CK)$ is finite.
\item[(2)]$\Sel'_\fL(A/\CK) \simeq (\Sel'_\fL(A/\CF))^\Delta$.
\item[(3)]Let $\CF_n=\CK(A_{\fL^{n+1}})$ for each integer $n\geq0$, and set $\CF_\infty=\CK(A(\fL))$. Define
\[\Sel'_\fL(A/\CF_\infty)=\lim_n \Sel'_\fL(A/\CF_n),\]
where the direct limit is taken with respect to the restriction maps. Let $\Gamma={\rm Gal}(\CF_\infty/\CF)$. Then
\[\Sel'_\fL(A/\CF)\simeq \Sel'_\fL(A/\CF_\infty)^\Gamma.\]
\end{enumerate}
\end{lem} 

All claims in Lemma \ref{ap2-lem3} are standard; see \cite{coates1} (or \cite[Theorem 5.1]{KL-lms} and references therein) for detailed proofs. Let $\CG={\rm Gal}(\CF_\infty/\CK)$ and $G={\rm Gal}(\CK/K)$. Note that $\Gamma={\rm Gal}(\CF_\infty/\CF)\simeq\BZ_\ell$.  Since $\ell$ is coprime to $[\CF:K]$ and unramified in $\CK/K$, we have 
$${\rm Gal}(\CF_\infty/K)\simeq G\times \CG.$$  
Similarly, $\CG\simeq\Gamma\times\Delta$.  The following proposition will complete the proof of Proposition \ref{ap2-prop2}

\begin{prop}\label{ap2-prop3}
For each character $\eta$ of $G$, if $L(\overline{\psi_A\eta},1)\neq0$, then $(\Sel'_\fL(A/\CF_\infty)^\CG)^\eta$ is finite.
\end{prop}

\begin{proof}
Let $X_A(\fL)$ denote the Pontryagin dual of $\Sel'_\fL(A/\CF_\infty)$. For a character $\eta$ of $G$, let $X_A(\fL)^\eta$ be its $\eta$-part. Let $\Lambda(\CG)$ denote the Iwasawa algebra of $\CG$.
By a theorem of Coates (see \cite{coates1}), $X_A(\fL)^\eta$ is a finitely generated torsion $\Lambda(\CG)$-module. Let $f_{A,\eta}$ be a generator of the characteristic ideal of $X_A(\fL)^\eta$. Let $\msi$ be the completion of the ring of integers in the maximal unramified extension of $\BZ_\ell$, and let $\Lambda_\msi(\CG)=\Lambda(\CG)\widehat{\otimes}\msi$. Denote by $\kappa$ the character of $\CG$ giving its action on $A(\fL)$.
By the elliptic analogue of Iwasawa's theorem for cyclotomic fields (this is due to Coates-Wiles, see \cite[Theorem 3.6]{rubin88} or \cite[Theorem 4.12]{deshalit}) together with Rubin's proof of the Iwasawa main conjecture \cite{rubin}, 
the ideal $f_{A,\eta}\Lambda_\msi(\CG)$ admits a generator $g_{\fL,\eta}$ satisfying 
$\kappa(g_{\fL,\eta})=\left(1-\frac{\psi_A\eta(\fL)}{\ell}\right)\cdot\frac{L(\overline{\psi_A\eta},1)}{\Omega}$,
where $\Omega$ is an $\CO_K$-generator of the period lattice of a minimal model of $A$ over $K$. The hypothesis $L(\overline{\psi_A\eta},1)\neq0$ thus implies $\kappa(f_{A,\eta})\neq0$. A standard result in Iwasawa theory (see \cite[Lemma 5.2]{KL-lms}) then shows that $(X_A(\fL)^\eta)_\CG$ is finite. By duality, this is equivalent to the finiteness of $(\Sel'_\fL(A/\CF_\infty)^\CG)^\eta$. 
\end{proof}

\section{Appendix C: Numerical data for Theorem \ref{thm2}}\label{sec9}

We list the first few examples (within 5000) of pairs of primes \((p,q)\) satisfying the conditions in Theorem \ref{thm2}:

\begin{center}
\begin{scriptsize}
\begin{table}[H]
\renewcommand{\arraystretch}{1.4}
\center
\begin{tabular}{|c|c|c|c|c|c|}
    \hline \(m\) & \(a_m\) & \(p\) & \(q\) & \(h(-q)\) & \(h(-6pq)\) \\
\hline 217 & -35 & 31 & 7 & 1 & 16 \\
\hline 721 & -65 & 103 & 7 & 1 & 64 \\
\hline 889 & -95 & 7 & 127 & 5 & 80 \\
\hline 1561 & 25 & 223 & 7 & 1 & 80 \\
\hline 1897 & 145 & 271 & 7 & 1 & 64 \\
\hline 2569 & 175 & 367 & 7 & 1 & 80 \\
\hline 2881 & -65 & 67 & 43 & 1 & 112 \\
\hline 3193 & 91 & 31 & 103 & 5 & 112 \\
\hline 3661 & -215 & 523 & 7 & 1 & 160 \\
\hline 4249 & 145 & 607 & 7 & 1 & 176 \\
\hline 4333 & 85 & 619 & 7 & 1 & 64 \\
\hline 4417 & -5 & 7 & 631 & 13 & 80 \\
\hline 4429 & 169 & 43 & 103 & 5 & 80 \\
\hline 4837 & 205 & 691 & 7 & 1 & 80 \\

    \hline
\end{tabular}
\end{table}
\end{scriptsize}
\end{center}

\end{document}